\documentclass{amsart}

\usepackage[english]{babel}
\usepackage[latin1]{inputenc}
\usepackage{amsmath}
\usepackage{amsthm}
\usepackage{amssymb}
\usepackage{tikz}

\usepackage{imakeidx}
\usepackage{appendix}
\usepackage{tikz-cd}
\usepackage{verbatim}
\usepackage{enumitem}
\usepackage{multicol}

\usepackage{hyperref}
\usepackage{cleveref}

\usepackage{todonotes}

\usepackage{mathtools}

\newtheorem{thm}{Theorem}[section]
\newtheorem{prop}[thm]{Proposition}

\newtheorem{lemma}[thm]{Lemma}
\newtheorem{cor}[thm]{Corollary}

\theoremstyle{definition}
\newtheorem{defin}[thm]{Definition}

\theoremstyle{remark}

\newtheorem{rmk}[thm]{Remark}
\numberwithin{equation}{section}

\newcommand{\R}{\mathbb R}
\newcommand{\C}{\mathbb C}
\newcommand{\Z}{\mathbb Z}
\newcommand{\G}{\mathbb G}

\renewcommand{\c}{\subseteq}
\renewcommand{\O}{\mathbb O}

\newcommand{\mc}[1]{\mathcal{#1}}
\newcommand{\cl}{\overline}
\newcommand{\set}[1]{\left\{#1\right\}}
\renewcommand{\phi}{\varphi}

\newcommand{\on}[1]{\operatorname{#1}}
\newcommand{\ang}[1]{\left \langle{#1}\right \rangle}

	\author{Andrew Fiori}
	\address{Department of Mathematics \& Computer Science\\
	University of Lethbridge\\
	Lethbridge, AB T1K 3M4\\Canada}
	\email{andrew.fiori@uleth.ca}
	\thanks{}

	\author{Federico Scavia}
\address{Department of Mathematics\\
	University of British Columbia\\
	Vancouver, BC V6T 1Z2\\Canada}
\email{scavia@math.ubc.ca}
\thanks{Federico Scavia was partially supported by a graduate fellowship from the University of British Columbia.}



\title{Embeddings of maximal tori in groups of type $F_4$}

\begin{document}
\maketitle
\begin{abstract}
    We classify maximal tori in groups of type $F_4$ over a local or global field of characteristic different from $2$ and $3$. We prove a local-global principle for embeddings of maximal tori in groups of type $F_4$.
\end{abstract}

\section{Introduction}
Let $k$ be a field, and let $G$ be a semisimple group over $k$ (not necessarily $k$-split). It is a natural problem to classify the isomorphism classes of maximal $k$-tori of $G$. 

Assume that $G$ is a classical group, and let $(A,\tau)$ be the associated central simple algebra with involution. Then it is well known that maximal tori in $G$ are associated to certain \'etale algebras with involution $(E,\sigma)$ inside $(A,\tau)$. 
However, to find necessary and sufficient conditions for the existence of an embedding of $(E,\sigma)$ in $(A,\tau)$ can be quite difficult, especially without restrictive assumptions on $k$.

The prototypical situation is that of special orthogonal groups. This case has been studied by several authors, and has found applications to many arithmetic problems, for example to special points in symmetric spaces; see e.g. \cite{brusamarello2003}, \cite{fiori2012characterization}, \cite{gille2004type}, \cite{prasad2010local}, \cite{bayer2014embeddings}, \cite{fiori2019rational}. In the majority of these works, $k$ is assumed to be a local or global field: on the one hand, this allows to obtain clear and explicit results, and on the other, such fields are the most important in many applications.

Assume that $k$ is a global field such that $\on{char}k\neq 2$, let $(V,q)$ be a quadratic $k$-space of even dimension $2n$, and let $G=\on{SO}(V,q)$. 
Let $(E,\sigma)$ be an \'etale algebra with involution such that $\on{rank}_kE=2\on{rank}_kE^{\sigma}=2n$, and let $U(E,\sigma)$ be the $k$-torus of rank $n$ defined by
\[U(E,\sigma)(R):=\set{x\in (E\otimes_kR)^{\times}: x\sigma(x)=1}\] for every $k$-algebra $R$. 

By \cite[Proposition 1.2.1]{bayer2014embeddings}, if $T$ is maximal $k$-torus in $\on{SO}(V,q)$, there exists a unique \'etale subalgebra $E'$ of $\on{End}(V)$ of rank $2n$ that is stable under the adjoint involution $\tau_q:\on{End}(V)\to \on{End}(V)$ such that $T=U(E',\tau_q|_{E'})$; moreover $\on{rank}_kE'=2\on{rank}_k(E')^{\tau_q}=2n$; see \Cref{maxtori-etalesub} for more details. If $(E,\sigma)$ is an \'etale algebra with involution such that $\on{rank}_kE=2\on{rank}_kE^{\sigma}=2n$, we say that $(E,q)$ is \emph{realizable} if there exists a maximal torus $T$ of $\on{SO}(V,q)$ such that the associated $(E',\tau_q)$ is isomorphic to $(E,\sigma)$, and we say that $T$ is of \emph{type} $(E,\sigma)$; see \cite[\S 1.2]{bayer2014embeddings}.

The problem is then to give necessary and sufficient conditions on $(E,\sigma)$ and $(V,q)$ for the realizability of $(E,q)$. The case of local fields is easier, so it seems natural to study the local case first, and then to attack the global case by proving a local-global principle. The local-global question that must be answered is the following: if $(E_v,q_v)$ is realizable over $k_v$ for every place $v$ of $k$, is it true that $(E,q)$ is realizable over $k$?

It turns out that the answer to this question is negative, as shown by a counterexample of G. Prasad and A. Rapinchuk \cite[Example 7.5]{prasad2010local}. In \cite{fiori2012characterization}, the first author gave necessary and sufficient conditions for the everywhere local realizability of $(E,q)$, that is, he answered the question of when $(E_v,q_v)$ is realizable for every place $v$. Moreover, he proved the local-global principle in the case when $E$ is a field. The failure of the local-global principle was later completely clarified by E. Bayer-Fluckiger \cite{bayer2014embeddings}. In subsequent papers \cite{bayer2015embedding}, \cite{bayer2016embeddings}, \cite{bayer2018embeddings}, E. Bayer-Fluckiger, T.-Y. Lee and R. Parimala more generally studied embeddings of maximal tori in classical groups.

We turn to the situation for exceptional groups. If $G$ is of type $G_2$, embeddings of maximal tori in $G$ have been studied in \cite{beli2015maximal} and \cite{hooda2018embeddings}. For outer forms of type $D_4$ (in particular, triality forms), classifications and local-global principles in many cases have been established in \cite{fiori2019rational}. 

In the present work, we analyze the case when $G$ is of type $F_4$ over a global field $k$ such that $\on{char}k\neq 2,3$. In analogy with the case of orthogonal groups, we parametrize the tori which embed in some group of type $F_4$ using \'etale algebras with extra structure. More precisely, in \Cref{relation} we show that maximal tori in groups of type $F_4$ can be constructed from a \emph{datum} $\alpha=(L,E,\sigma,\beta)$, where $L$ is a cubic \'etale $k$-algebra, $(E,\sigma)$ is an \'etale $L$-algebra with involution such that $\on{rank}_LE=2\on{rank}_LE^{\sigma}=8$ and $\beta$ is an isomorphism of \'etale $L$-algebras defined in (\ref{beta-eq}). We say that the pair $(G,\alpha)$ is \emph{realizable} if there exists a maximal torus $T$ of $G$ with associated datum $\alpha$. We give a functorial description of these tori in terms of $\alpha$ which is entirely analogous to that of tori $U(E,\sigma)$; see (\ref{intrinsicF4}). 

Our main result is a proof of the local-global principle for maximal tori inside $G$. 
\begin{thm}\label{mainthm}
Let $k$ be a global field of characteristic different from $2$ and $3$, let $G$ be a $k$-group of type $F_4$, and let $\alpha$ be a datum. Then $(G,\alpha)$ is realizable if and only if $(G_v,\alpha_v)$ is realizable for every place $v$ of $k$.
\end{thm}
The validity of the local-global principle is quite surprising, especially because our proof relies on the results for orthogonal groups of type $D_4$, for which the local-global principle does not hold, in general. More precisely, our proof of \Cref{mainthm} uses the following new result on the local-global principle for embeddings of tori in orthogonal groups. Let $(V,q)$ be a quadratic space of even dimension over a global field $k$, let $(E,\sigma)$ be an \'etale algebra with involution over $k$, and assume that $(E_v,q_v)$ is realizable for every place $v$ of $k$. As we have said, this does not imply that $(E,q)$ is realizable. However, in \Cref{trivialclifford} we show that if $q$ has trivial Clifford invariant, then $(E,q)$ is realizable. In other words, the local-global principle for embeddings of maximal tori in orthogonal groups with trivial Clifford invariant holds. 

Using \Cref{mainthm} and a local analysis, we are able to classify maximal tori for all groups of type $F_4$ over global fields. Recall that there are exactly three real forms of type $F_4$: the split form, the anisotropic form, and the form of real rank $1$. 

If $(E,\sigma)$ is an \'etale algebra with involution such that $\on{rank}_kE=2\on{rank}_kE^{\sigma}=2n$, we say that a real place $w$ of $E^{\sigma}$ is ramified if it extends to a complex place of $E$. We denote by $\rho_v$ the number of real places of $E^{\sigma}$ above $v$ which are not ramified in $E$. We have $0\leq \rho_v\leq n$.

\begin{thm}\label{classification}
Let $k$ be a global field of characteristic different from $2$ and $3$, let $G$ be a group of type $F_4$, and let $\alpha=(L,E,\sigma,\beta)$ be a datum. Then $(G,\alpha)$ is realizable if and only if for every real place $v$ of $k$ one of the following holds:
\begin{enumerate}[label=(\roman*)]
    \item\label{classification1} $G_v$ is split and either $L_v=\R^3$ and $\rho_v$ is even, or $L_v=\R\times\C$ and $\rho_v$ is odd;
    \item\label{classification2} $G_v$ is anisotropic, $L_v=\R^3$ and $\rho_v=0$; 
    \item\label{classification3} $G_v$ has real rank $1$ and either $L_v=\R^3$ and $\rho_v$ is even, or $L_v=\R\times \C$ and $\rho_v=1$.
\end{enumerate}
\end{thm}
It is not difficult to describe the possible isomorphism classes of maximal tori in $G_v$, for every real place $v$; see \Cref{realtori}.

We conclude the paper with a brief discussion of the topic of rational conjugacy classes of  tori in groups of type $F_4$; this can be found in \Cref{sec:conjclasses}.

\section{Construction of $H$ - the split case}\label{constructionH}
Our main references for octonion algebras, twisted compositions and Albert algebras are \cite{springer2000octonions} and \cite{knus1998involutions}.

Let $k$ be a field. We assume that $\on{char}k\neq 2,3$. We denote by $\mathbb{O}$ the split octonion algebra over $k$. As a $k$-vector space, $\mathbb{O}=M_2(k)^{\oplus 2}$. If
\[A=\begin{pmatrix} a_{11} & a_{12} \\ a_{21} & a_{22}\end{pmatrix}\in M_2(k)\]
we define 
\[\cl{A}=\begin{pmatrix} a_{22} & -a_{12} \\ -a_{21} & a_{11},\end{pmatrix}\]
so that $A\cl{A}=\cl{A}A=\det(A)\cdot I$. The product on $\mathbb{O}$ is defined as \[(A,B)(C,D):=(AC+\cl{D}B,DA+B\cl{C}).\]
The norm on $\mathbb{O}$ is defined as $n(A,B)=\on{det}(A)-\on{det}(B)$. We denote by $\ang{\cdot,\cdot}$  the associated inner product on $\mathbb{O}$.

Let $G_s$ be the split group of type $F_4$. We have $G_s=\on{Aut}(J_s)$, where $J_s\c M_3(\O)$ is the $k$-split Albert algebra:
\[J_s:=\set{\begin{pmatrix} \xi_1 & c_3 & \cl{c_2} \\ \cl{c_3} & \xi_2 & c_1 \\ c_2 & \cl{c_1} & \xi_3 \end{pmatrix}: \xi_i\in k, c_i\in \O}.\]
We denote by $\cdot$ the matrix product in $M_3(\O)$, and by juxtaposition the product in $J_s$, that is \[xy=\frac{1}{2}(x\cdot y+y\cdot x)\] for every $x,y\in J_s$.

Denote by $Q$ the norm of $J_s$; see \cite[(5.3)]{springer2000octonions}. It is a non-degenerate quadratic form on $J_s$. Let $\ang{\cdot,\cdot}$ be the bilinear form associated to $Q$, called the inner product of $J_s$. We denote by $e$ the identity element of $J_s$. Let $\on{det}$ be the determinant form of $J_s$, as defined in \cite[(5.7)]{springer2000octonions}: it is a cubic form on $J_s$. Denote by $\ang{\cdot,\cdot,\cdot}$ the trilinear form associated to $\on{det}$. For every $x,y\in J_s$, define the cross product $x\times y$ as the unique element of $J_s$ such that 
\[\ang{x\times y,z}=3\ang{x,y,z}\]
for every $z\in J_s$; see \cite[\S 5.2]{springer2000octonions}.

We set
\[L_s:=\set{\begin{pmatrix} \xi_1 & 0 & 0 \\ 0 & \xi_2 & 0 \\ 0 & 0 & \xi_3 \end{pmatrix}: \xi_i\in k}\] 
and
\[M_s:=\set{\begin{pmatrix} 0 & c_3 & \cl{c_2} \\ \cl{c_3} & 0 & c_1 \\ c_2 & \cl{c_1} & 0 \end{pmatrix}: c_i \in \O}.\]
It is clear that $L_s$ is a subalgebra of $J_s$. We will give $M_s$ the structure of an $L_s$-module; see \Cref{formulas}.
We have $M_s=L_s^{\perp}$. An element of $M_s$ may be written as a triple $(c_1,c_2,c_3)$, where $c_i\in\mathbb{O}$. This yields a direct sum decomposition $M_s=J_1\oplus J_2\oplus J_3$, where $J_i\cong \O$ and $c_i\in J_i$. Thus \[J_s=L_s\oplus M_s=L_s\oplus J_1\oplus J_2\oplus J_3.\]
It follows immediately from the definition of $Q$ that the restriction of $Q$ to the $J_i$ is the octonion norm.

The subgroup of $G_s$ fixing every element of $L_s$ is isomorphic to $\on{Spin}_8$, and the subgroup of automorphisms preserving $L_s$ is $\on{Spin}_8\rtimes S_3$; see \cite[\S 39.19]{knus1998involutions}. We denote by $H_s\cong \on{Spin}_8$ its connected component. By \cite[\S 35.8]{knus1998involutions}, for every $k$-algebra $R$, we have \[H_s(R)=\set{(g_1,g_2,g_3)\in \on{SO}_8^3(R): g_1(x\cdot y)=\cl{g_3(\cl{x})}\cdot\cl{g_2(\cl{x})} \text{ for all }x,y\in \O\otimes_kR}.\]

We now construct a $k$-split maximal torus $T_s$ of $G_s$, following the exposition of \cite{macdonald2014essential}. 

For every $\lambda\in k^{\times}$, let $R_{\lambda}$ and $L_\lambda$ be the linear automorphisms of $\O$ given by right and left multiplication by
\[\begin{pmatrix} \lambda & 0 \\ 0 & \lambda^{-1}\end{pmatrix},\] respectively. Consider the following elements of $\on{SO}_8^3(k)$:
\[s_\lambda:=(R_\lambda,L_{\lambda^{-1}}\circ R_{\lambda^{-1}},L_\lambda), \qquad r_{\lambda}:=(L_\lambda,R_\lambda,L_{\lambda^{-1}}\circ R_{\lambda^{-1}}).\]
One can check that $r_{\lambda},s_{\lambda}\in \on{Spin}_8(R)$, and that the collection $\set{r_{\lambda},s_{\lambda}}_{\lambda\in k^{\times}}$ generates a $2$-dimensional split torus $T_1$; see \cite[1.6]{macdonald2014essential}.\footnote{More precisely, one should define $r_{\lambda}$ and $s_{\lambda}$ for every $\lambda\in R^{\times}$ and every $k$-algebra $R$, and then check that this defines a $k$-group homomorphism $\phi:\G_m^2\to G$. We let $T_1$ be the image of $\phi$.}

The group $\on{Spin}_8$ has a $k$-split subgroup of type $G_2$. Its $k$-points may be described as the triples $(g_1,g_2,g_3)\in \on{SO}_8^3(k)$ such that $g_i(1)=1$; see \cite[\S 35.16]{knus1998involutions}. It is clear from this description that $G_2$ and $T_1$ intersect trivially. The group $G_2$ contains a subgroup isomorphic to $\on{SL}_3$; it is the stabilizer of \[i:=(\begin{pmatrix} 1 & 0 \\ 0 & -1\end{pmatrix},0),\] in the $7$-dimensional representation of $G_2$ (trace-zero octonions); see \cite[\S 36, Exercise 6]{knus1998involutions} or \cite[\S 2.2]{springer2000octonions}. Denote by $T_2$ the maximal torus inside $\on{SL}_3$ constructed in \cite[\S 2.3]{springer2000octonions}. The tori $T_1$ and $G_2$ intersect trivially and commute; see \cite[1.7]{macdonald2014essential}. We define $T_s$ as the torus generated by $T_1$ and $T_2$: it is a $k$-split maximal torus of $G_s$. The torus $T_s$ acts trivially on $L_s$, and its action on $M_s=L_s^{\perp}$ has no non-zero fixed points.

We now want to define on $(L_s,M_s)$ a structure of a split twisted composition over $k$; see \cite[\S 36]{knus1998involutions} for the definition. In the next paragraph we will carry out a similar construction for an arbitrary Albert algebra $J$. For this reason, it is useful to give an intrinsic definition first, and only later write it in coordinates.

Using multiplication in $J_s$ and projections, one may construct $k$-bilinear maps
\begin{equation}\label{maps}\mu_{L_s}:L_s\times L_s\to L_s,\qquad \mu_{M_s}:M_s\times M_s\to M_s.\end{equation}
For example, $\mu_{L_s}$ is the composition
\[L_s\times L_s\hookrightarrow J_s\times J_s\to J_s\xrightarrow{\pi_s} L_s\] where the second map is multiplication in $J_s$ and $\pi_s$ is projection onto $L_s$; a similar construction gives $\mu_{M_s}$.
There is also a $k$-bilinear map
\begin{align}\eta_s:L_s\times M_s&\to M_s\\
(\xi,c)&\mapsto -2\xi c+\ang{\xi,e}c.\nonumber
\end{align}
One could also define $\eta_s$ using the cross product on $J_s$; see \cite[p. 162]{springer2000octonions}. That $\eta_s$ really takes values in $M_s$ will be shown in \Cref{formulas}.

Finally, we have a $k$-bilinear map 
\[\phi_s:M_s\times M_s\to L_s,\] given by the composition
\[M_s\times M_s\hookrightarrow J_s\times J_s\xrightarrow{\times} J_s\xrightarrow{\pi_s} L_s.\]
We define $q_s:=-2\phi_s$. If $\xi\in L_s$ and $c\in M_s$, we express them in coordinates as $\xi=(\xi_1,\xi_2,\xi_3)$ and $c=(c_1,c_2,c_3)$. 

\begin{lemma}\label{formulas}
Let $\xi,\xi'\in L_s$ and $c,c'\in M_s$. The following identities hold:
\begin{align}\mu_{L_s}(\xi,\xi')=(\xi_1,\xi_2,\xi_3)(\xi'_1,\xi'_2,\xi'_3)=(\xi_1\xi_1',\xi_2\xi_2',\xi_3\xi_3')\in L_s.\label{formulas1}\\
\eta_s(\xi,c)=-2\xi c+\ang{\xi,e}c=(\xi_1c_1,\xi_2c_2,\xi_3c_3)\in M_s.\label{formulas2}\\
q_s(c,c')=(\ang{c_1,c'_1},\ang{c_2,c'_2},\ang{c_3,c'_3})\in L_s.\label{formulas3}\\
\mu_{M_s}(c,c)=(\cl{c_2c_3}, \cl{c_3c_1},\cl{c_1c_2})\in M_s.\label{formulas4}
\end{align}
\end{lemma}

\begin{proof}
The verification of (\ref{formulas1}) is immediate.

We have
\[\begin{pmatrix} \xi_1 & 0 & 0 \\ 0 & \xi_2 & 0 \\ 0 & 0 & \xi_3 \end{pmatrix}\cdot\begin{pmatrix} 0 & c_3 & \cl{c_2} \\ \cl{c_3} & 0 & c_1 \\ c_2 & \cl{c_1} & 0 \end{pmatrix}=\begin{pmatrix} 0 & \xi_1c_3 & \xi_1\cl{c_2} \\ \xi_2\cl{c_3} & 0 & \xi_2c_1 \\ \xi_3c_2 & \xi_3\cl{c_1} & 0 \end{pmatrix}\]
and
\[\begin{pmatrix} 0 & c_3 & \cl{c_2} \\ \cl{c_3} & 0 & c_1 \\ c_2 & \cl{c_1} & 0 \end{pmatrix}\cdot \begin{pmatrix} \xi_1 & 0 & 0 \\ 0 & \xi_2 & 0 \\ 0 & 0 & \xi_3 \end{pmatrix}=\begin{pmatrix} 0 & \xi_2c_3 & \xi_3\cl{c_2} \\ \xi_1\cl{c_3} & 0 & \xi_3c_1 \\ \xi_1c_2 & \xi_2\cl{c_1} & 0 \end{pmatrix}.\]
Hence
\begin{equation}\label{construction-eq2}\xi c=\frac{1}{2}((\xi_2+\xi_3)c_1, (\xi_1+\xi_3)c_2, (\xi_1+\xi_2)c_3)\in M_s.\end{equation}
Using that $\ang{\xi,e}=\xi_1+\xi_2+\xi_3$, one concludes that \[\eta_s(\xi,c)=-2\xi c+\ang{\xi,e}c=(\xi_1c_1,\xi_2c_2,\xi_3c_3)\in M_s,\] proving (\ref{formulas2}).

By \cite[Lemma 5.2.1(i)]{springer2000octonions} we have \[x\times y=xy-\frac{1}{2}\ang{x,e}y-\frac{1}{2}\ang{y,e}x-\frac{1}{2}\ang{x,y}e+\frac{1}{2}\ang{x,e}\ang{y,e}e\] for every $x,y\in J$. Using this formula together with \begin{equation}\label{construction-eq}\begin{pmatrix} 0 & c_3 & \cl{c_2} \\ \cl{c_3} & 0 & c_1 \\ c_2 & \cl{c_1} & 0 \end{pmatrix}\cdot
\begin{pmatrix} 0 & c_3' & \cl{c_2}' \\ \cl{c_3}' & 0 & c'_1 \\ c'_2 & \cl{c_1}' & 0 \end{pmatrix}=
\begin{pmatrix} c_3\cl{c_3}'+\cl{c_2}c'_2 & \cl{c_2}\cl{c_1}' & c_3c'_1 \\ c_1c'_2 & \cl{c_3}c'_3+c_1\cl{c_1}' & \cl{c_3}\cl{c_2}' \\ \cl{c_1}\cl{c_3}' & c_2c'_3 & c_2\cl{c_2}'+\cl{c_1}c'_1 \end{pmatrix}.\end{equation} one can verify without difficulty that \[\phi(c,c')=-\frac{1}{2}(\ang{c_1,c'_1},\ang{c_2,c'_2},\ang{c_3,c'_3});\]
 see \cite[p. 163]{springer2000octonions}. This proves (\ref{formulas3}).

From (\ref{construction-eq}), we have that \[\mu_{M_s}((c_1,c_2,c_3),(c'_1,c'_2,c'_3))=\frac{1}{2}(\cl{c_3}\cl{c_2}'+\cl{c_3}'\cl{c_2},\cl{c_1}\cl{c_3}'+\cl{c_1}'\cl{c_3},\cl{c_2}\cl{c_1}'+\cl{c_2}'\cl{c_1}).\] It follows that \[\mu_{M_s}(c,c)=(\cl{c_2c_3}, \cl{c_3c_1},\cl{c_1c_2}).\qedhere\] 
\end{proof}

\begin{prop}\label{splitcomp}
Let $x^{*2}:=\mu_{M_s}(x,x)$. Together with the maps previosly defined, the quadruple $(L_s,M_s,q,\cdot^{*2})$ is a split twisted composition.
\end{prop}

\begin{proof}
This is a consequence of \Cref{formulas}. By (\ref{formulas1}), we have an isomorphism $L_s\cong k^3$ of $k$-algebras. By (\ref{formulas2}), the map $\eta_s$ makes $M_s$ into an $L_s$-module, and we have $\on{dim}_{L_s}M_s=8$. By (\ref{formulas3}), $q_s$ is a non-degenerate symmetric $L_s$-bilinear form. The associated quadratic form is \begin{equation}\label{construction-eq5}q_s(c)=(N(c_1),N(c_2),N(c_3)),\end{equation} where $N:\O\to k$ is the octonion norm. To prove that $(L_s,M_s,q,\cdot^{*2})$ is a split composition, it remains to show that
\begin{equation}\label{construction-eq3} q_s(c)q_s(c^{*2})=N_{L_s/k}(q_s(c))\end{equation}
and \begin{equation}\label{construction-eq4}
    \xi(\xi c)^{*2}=N_{L_s/k}(\xi)c^{*2}
\end{equation} for every $\xi\in L_s, c\in M_s$. Here $N_{L_s/k}(\xi)=\xi_1\xi_2\xi_3\in k$.

We have $N(\cl{x})=N(x)$ and $N(xy)=N(x)N(y)$ for every $x,y\in \O$, hence by (\ref{construction-eq5})
\[q_s(c^{*2})=(N(c_2)N(c_3),N(c_1)N(c_3),N(c_1)N(c_2)).\] We deduce that
\[q_s(c)q_s(c^{*2})=N(c_1)N(c_2)N(c_3)=N_{L_s/k}(q_s(c)),\] which proves (\ref{construction-eq3}). The proof of (\ref{construction-eq4}) is similar.
\end{proof}

\begin{lemma}\label{automorphism}
    The groups $H_s$ and $T_s$ act on $J_s$ via automorphisms of the twisted composition $(L_s,M_s,q_s,\cdot^{*2})$.
\end{lemma}

\begin{proof}
Since $T_s$ is contained in $H_s$, it is enough to prove the claim for $H_s$. By definition, $H_s$ fixes $L_s$ pointwise, the subspaces $J_1,J_2,J_3$ are $H_s$-stable, and the $H_s$-action on each $J_i\cong \O$ respects the octonion norm, so
we only need to show that for every $h\in H_s(\cl{k})$ and every $m\in M_s(\cl{k})$ we have $h(m^{*2})=h(m)^{*2}$. Let $m\in M_s(\cl{k})$ and $h\in H_s(\cl{k})$, and write \[m^2=l+n,\qquad (hm)^2=l'+n',\] where $l,l'\in L_s(\cl{k})$ and $n,n'\in M_s(\cl{k})$. By definition, we have \[m^{*2}=n,\qquad h(m^{*2})=h(n),\qquad (hm)^{*2}=n'.\] Since $H_s$ acts on $J_s$ via Albert algebra automorphisms, we have $h(m^2)=(hm)^2$. Since $J_s=L_s\oplus M_s$, we deduce that $h(n)=n'$. This means that $h(m^{*2})=(hm)^{*2}$, as desired.
\end{proof}

We have constructed a diagram $T_s\hookrightarrow H_s\hookrightarrow G_s$. In the next section, we will generalize this to maximal tori in arbitrary groups of type $F_4$.

\section{Construction of $H$ - the general case}\label{Hconstr}
Let $J$ be an Albert algebra over $k$. Similarly to the split case, we denote by $Q$ the norm of $J$; it is a non-degenerate quadratic form on $J$. We let $\ang{\cdot,\cdot}$ be the bilinear form associated to $Q$, called the inner product of $J$, and we denote by $e$ the identity element of $J$.

We let $\on{det}$ be the determinant form of $J$, and we denote by $\ang{\cdot,\cdot,\cdot}$ the trilinear form associated to $\on{det}$. For every $x,y\in J$, we define the cross product $x\times y$ as the unique element of $J$ such that 
\[\ang{x\times y,z}=3\ang{x,y,z}\]
for every $z\in J$.

Let $G$ be a group of type $F_4$ over $k$, and let $J$ be the Albert algebra associated to $G$, so that $G=\on{Aut}(J)$. By a Galois descent argument, every $k$-group of type $F_4$ arises in this way; see \cite[Proposition 37.11]{knus1998involutions}.

 Let $T$ be a maximal $k$-torus of $G$ (not necessarily $k$-split). Since $G_{\cl{k}}$ is split, we have a constructed in \Cref{constructionH} a maximal torus $T_s$ and a group $H_s\cong \on{Spin}_8$ such that $T_s\hookrightarrow H_s\hookrightarrow G_{\cl{k}}$. There exists $g\in G(\cl{k})$ such that $gT_{\cl{k}}g^{-1}=T_s$. The torus $T$ acts on $J$ via $G$. If we let $L$ be the $T$-fixed subspace of $J$, then $gL_{\cl{k}}g^{-1}=L_s$, hence $L$ is an \'etale algebra of degree $3$ over $k$. If we set $M:=L^{\perp}$, we have $J=L\perp M$ and $gM_{\cl{k}}g^{-1}=M_s$, hence the $T$-action on $M$ has no non-zero fixed points.

We now define $k$-bilinear maps $\mu_L$, $\mu_M$, $\eta$, $q$, mimicking the construction of $\mu_{L_s}$, $\mu_{M_s}$, $\eta_s$ and $q_s$. Using multiplication in $J$ and projections, we construct
\begin{equation}\label{maps2}\mu_L:L\times L\to L,\qquad \mu_M:M\times M\to M.\end{equation}
We define
\begin{align}\eta:L\times M&\to M\\
(l,m)&\mapsto -2lm+\ang{l,e}m.\nonumber
\end{align}
As in the split case, one could also define $\eta$ using the cross product $\times$ on $J$; see \cite[p. 162]{springer2000octonions}. 

Finally, we have a $k$-bilinear map 
\[\phi:M\times M\to L,\] given by the composition
\[M\times M\hookrightarrow J\times J\xrightarrow{\times} J\xrightarrow{\pi} L,\] where $\pi$ denotes the projection onto $L$. We define $q:=-2\phi$.

The decomposition $J=L\perp M$ is known as the \emph{Springer decomposition}; see \cite[38.A]{knus1998involutions}.

\begin{prop}\label{constr}
Let $G$ be a group of type $F_4$, not necessarily $k$-split, and let $T$ be a maximal torus of $G$. Define $x^{*2}:=\mu_M(x,x)$.
\begin{enumerate}[label=(\alph*)]
    \item\label{constr1} Together with the previously defined maps, $\Gamma=(L,M,q,\cdot^{*2})$ is a twisted composition algebra over $k$, and $J=L\perp M$ is a Springer decomposition.
    \item\label{constr2} The $T$-action on $J$ induces an action on $M$ via automorphisms of the twisted composition.
    \item\label{constr3} The group $\on{Aut}(\Gamma)^{\circ}$ is simply connected of type $D_4$, it canonically embeds in $G$, and its image $H$ is the unique subgroup of $G$ of type $D_4$ containing $T$.
    \item\label{constr4} Every subgroup of $G$ of type $D_4$ is uniquely of the form $\on{Aut}(\Gamma')^{\circ}$, for some twisted composition $\Gamma'=(L',M',q',\cdot^{*2})$ such that $J=L'\perp M'$ is a Springer decomposition.
\end{enumerate} 
\end{prop}
    The quadratic form $M\to L$ associated to $q$ will also be denoted by $q$.
\begin{proof}
\ref{constr1} The maps are defined over $k$, and the properties that we want to prove can be checked on an algebraic closure of $k$, so we may assume that $k$ is algebraically closed and $J=J_s$. Up to conjugation, we may assume that $T=T_s$. Then $L=L_s$ and $M=M_s$. We conclude using \Cref{splitcomp}.

\ref{constr2} After passing to an algebraic closure of $k$, the statement follows immediately from \Cref{automorphism}.

\ref{constr3} We have $J=L\perp M$ by construction. By \cite[Proof of Corollary 38.7]{knus1998involutions}, an automorphism of $\Gamma$ extends uniquely to an automorphism of $J$ as an Albert algebra. It follows that $\on{Aut}(J)$ is canonically a subgroup of $G$, and in particular this is true for $\on{Aut}(J)^{\circ}$. The fact that $T\c H$ is now an immediate consequence of \ref{constr2}. On an algebraic closure of $k$, the torus $T$ is conjugate to $T_s$, say $gT_sg^{-1}=T$. The procedure used to construct $H$ from $T$ coincides with that used to construct $H_s$ on $T_s$, hence $gH_sg^{-1}=H$. Since $H_s\cong \on{Spin}_8$, we conclude that $H$ and $\on{Aut}(\Gamma)^{\circ}$ are simply connected of type $D_4$. 

To show the uniqueness of $H$, we may assume that $k$ is algebraically closed. The result is now a consequence of Borel-de Siebenthal theory, which classifies subgroups of maximal rank of semisimple groups; see \cite[p. 219]{borel1949sous-groupes}. In type $F_4$, there are three  maximal subgroups of maximal rank containing $T$, of type $B_4$, $A_2\times A_2$ and $A_1 \times C_3$. Using \cite[p. 219]{borel1949sous-groupes} again, we see that the last two do not contain subgroups of type $D_4$, and the first contains exactly one.

\ref{constr4} Let $H$ be a subgroup of $G$ of type $D_4$, and let $T$ is a maximal torus of $H$. Since $H$ and $G$ have the same rank, $T$ is a maximal torus of $H$. By \ref{constr3}, it follows that $H$ is the unique simply connected subgroup of $G$ of type $D_4$ containing $T$, and it is the connected component of the automorphism group of the twisted composition associated to $T$.
\end{proof}

\begin{rmk}
Let $G$ be a $k$-group of type $F_4$, let $T$ be a maximal $k$-torus of $G$, and let $\Phi(G,T)$ be the corresponding root system. Since the $\on{Gal}(k)$-action respects the length of a root, it is easy to see that the set of long roots of $\Phi(G,T)$ spans a $k$-group $H$ of type $D_4$ and such that $T\hookrightarrow H\hookrightarrow G$. One can check that if $G=G_s$ and $T=T_s$, then $H=H_s$, thus showing that $H$ is simply connected. We conclude that $H$ is the spin group associated to some trialitarian algebra. However, \Cref{constr} is a much more precise statement: for example, it shows that only spin groups associated to twisted compositions occur.
\end{rmk}

Starting from a twisted composition $\Gamma=(L,M,q,\cdot^{*2})$, there is a canonical way to construct an Albert algebra $J(L,M)$, explained in \cite[\S 6.3]{springer2000octonions}. Let $J:=L\oplus M$ as a $k$-vector space. One defines a quadratic form $Q:J\to k$ as in \cite[(6.17)]{springer2000octonions}, a crossed product $\times:J\times J\to J$ as in \cite[(6.19)]{springer2000octonions}, and product on $J$ as in \cite[(6.21)]{springer2000octonions}. By \cite[Theorem 6.3.2]{springer2000octonions}, this makes $J$ into an Albert algebra, which we denote by $J(L,M)$.

The algebra $J(L,M)$ is called the \emph{Springer construction} associated with $\Gamma$; see \cite[\S 38.A]{knus1998involutions} and in particular \cite[Theorem 38.6]{knus1998involutions}, where the construction is carried out in an equivalent way. The orthogonal decomposition $J(L,M)=L\perp M$ is a Springer decomposition of the Springer construction.

\begin{rmk}\label{springerconstr}
Let $J$ be a Jordan algebra over $k$, and let $\Gamma=(L,M,q,\cdot^{*2})$ be a twisted composition over $k$. Let $G:=\on{Aut}(J)$ and let $H:=\on{Aut}(J)^{\circ}$. If $J\cong J(L,M)$, by \cite[Proof of Corollary 38.7]{knus1998involutions}, an automorphism of $\Gamma$ uniquely extends to an automorphism of $J$ as an Albert algebra, hence $H$ is canonically a subgroup of $G$. 
\end{rmk}

\section{Preliminaries on maximal tori in orthogonal groups}
We set up the notation for quadratic spaces and \'etale algebras with involutions in accordance to \cite[\S 1]{bayer2014embeddings}. In this section, $k$ is a field of characteristic different from $2$.

\subsection{Quadratic spaces}
We say that $(V,q)$ is a quadratic space over $k$ if $V$ is a finite-dimensional $k$-vector space and $q:V\times V\to k$ is a non-degenerate symmetric bilinear form. We use the same symbol $q$ to denote the quadratic form associated with $q$. We denote by $\on{O}(V,q)$ or by $\on{O}(q)$ its orthogonal group. We let $\on{det}(q)\in k^{\times}/k^{\times 2}$ be the determinant of $q$, and we define the discriminant of $q$ as $\on{disc}(q)=(-1)^{\frac{m(m-1)}{2}}\on{det}(q)$, where $m:=\on{dim}(q)$. The quadratic form $q$ can be diagonalized, i.e. there exist $a_1,\dots,a_{m}\in k^{\times}$ such that $q\cong \ang{a_1,\dots,a_{m}}$. 

We denote the Brauer group $\on{Br}(k)$ of $k$, viewed as an abelian group, and we let $\on{Br}_2(k)$ be its $2$-torsion subgroup. If $a,b\in k^{\times}$, we denote by $(a,b)\in \on{Br}_2(k)$ the Brauer class of the quaternion algebra determined by $a,b$. The Hasse invariant of $q$ is by definition $w(q):=\sum_{i<j}(a_i,a_j)\in \on{Br}_2(k)$. 

We denote by $\tau_q:\on{End}(V)\to \on{End}(V)$ the adjoint involution of $q$, i.e. $q(f(x),y)=q(x,\tau_q(f)(y))$ for all $f\in\on{End}(V)$ and $x,y\in V$. If $q$ and $q'$ are quadratic spaces, we denote by $q\perp q'$ their orthogonal sum. 

If $k=\R$, we say that $q$ has signature $(r,s)$ if $q$ is isometric to $x_1^2+\dots+x_r^2-x_{r+1}^2-\dots-x_{r+s}^2$.

\subsection{Maximal tori and \'etale algebras with involution}
Let $E$ be an \'etale algebra over $k$, let $\sigma:E\to E$ be a $k$-linear involution, and denote by $E^{\sigma}$ the subalgebra of $E$ fixed by $\sigma$. By definition, the pair $(E,\sigma)$ is an \'etale algebra with involution. We will always assume that $\on{rank}_kE=2\on{rank}_kE^{\sigma}$. We define the $k$-torus $U(E,\sigma)$ by setting \[U(E,\sigma)(R):=\set{x\in E\otimes_kR: x\sigma(x)=1}\] for every commutative $k$-algebra $R$. 

It is possible to give an alternative description of $U(E,\sigma)$, using Weil restrictions. If $E=E_1\times E_2$, where $E_1$ and $E_2$ are $\sigma$-stable, it is clear that $U(E,\sigma)=U(E_1,\sigma|_{E_1})\times U(E_2,\sigma|_{E_2})$. It is thus sufficient to describe $U(E,\sigma)$ in the cases when (i) $E$ is a field and (ii) $E=F\times F$ and $\sigma$ switches the two field factors $F$. In case (i) $U(E,\sigma)=R_{E^{\sigma}/k}(R^{(1)}_{E/E^{\sigma}}(\G_m))$, and in case (ii) $U(E,\sigma)=R_{F/k}(\G_m)$.

Let $(V,q)$ be a quadratic space of dimension $2n$ over $k$. By \cite[Proposition 1.2.1(i)]{bayer2014embeddings}, every maximal torus $T$ of $\on{O}(V,q)$ is of the form $U(E,\tau_q)$, where $E\c \on{End}(V)$ is an \'etale algebra $\tau_q$-stable and $\on{rank}_kE=2\on{rank}_kE^{\sigma}=2n$; the \'etale subalgebra $E$ is unique. Conversely, by \cite[Proposition 1.2.1(ii)]{bayer2014embeddings}, for every \'etale algebra $E\c \on{End}(V)$ stable under $\tau_q$ and such that $\on{rank}_kE=2\on{rank}_kE^{\tau_q}=2n$, the torus $U(E,\tau_q)$ is a maximal $k$-torus of $\on{O}(V,q)$.

Let $(E,\sigma)$ be an \'etale algebra with involution. For every $\alpha\in E^{\sigma}$, let $q_{\alpha}:E\times E\to k$ be the non-degenerate symmetric bilinear form defined by $q_{\alpha}(x,y):=\on{Tr}_{E/k}(\alpha x\sigma(y))$. By \cite[Proposition 1.3.1]{bayer2014embeddings}, $U(E,\sigma)$ is a maximal subtorus of the orthogonal group $\on{O}(V,q)$ if and only if there exists $\alpha\in E^{\sigma}$ such that $(V,q)$ is isometric to $(E,q_{\alpha})$. When this is the case, we say that $(E,q)$ is \emph{realizable}, and that $T$ is of \emph{type} $(E,\sigma)$; see \cite[\S 1.2]{bayer2014embeddings}.

\section{Description of maximal tori of groups of type $F_4$}
By \Cref{constr}\ref{constr1}, every maximal torus in a group of type $F_4$ is contained inside the automorphism group of a twisted composition. Moreover, \Cref{constr}\ref{constr3} and \ref{constr4} give a complete understanding of groups of type $D_4$ inside a group of type $F_4$. To be able to parametrize tori in groups of type $F_4$, all that is left is the study of maximal tori in automorphism groups of twisted compositions.

\begin{lemma}\label{isogeny}
Let $H_1$ and $H_2$ be reductive groups, and let $\pi:H_1\to H_2$ be an isogeny. The map $\pi$ sets up a correspondence between the maximal tori in $H_1$ and the maximal tori in $H_2$. 
\end{lemma}

\begin{proof}
Let $T_1$ be a maximal $k$-torus of $H_1$, and let $T_2:=\pi(T_1)$; it is a maximal $k$-torus of $H_2$. Denote by $N_i$ the normalizer of $T_i$ in $H_i$, for $i=1,2$. The isogeny $\pi$ induces a surjection $N_1\to N_2$, because $H_1$ and $H_2$ have the same Weyl group scheme. We deduce that the induced map $\phi:H_1/N_1\to H_2/N_2$ between the varieties of maximal tori of $H_1$ and $H_2$ (see \cite[\S 4.1]{voskresenskii2011algebraic}) is an isomorphism. This concludes the proof. 
\end{proof}

Let $\Gamma=(L,M,q,\cdot^{*2})$ be a twisted composition, and let $H:=\on{Aut}(\Gamma)^{\circ}$. By \Cref{constr}\ref{constr3}, $H$ is a $k$-form of $\on{Spin}_8$, split by $L$. We have an embedding \begin{equation}\label{triality}
    \iota:H\hookrightarrow R_{L/k}(\on{O}(M,q)),
\end{equation}  
that is constructed as follows. We have an embedding $j:H\hookrightarrow {R}_{L/k}(\on{Spin}(M,q))$; see \cite[\S 44.B p. 561]{knus1998involutions} or \cite[\S 3.4 p. 12]{fiori2019rational}. Then $\iota$ is the composition of $j$ with the natural projection ${R}_{L/k}(\on{Spin}(M,q))\to R_{L/k}(\on{O}(M,q))$. We call $\iota$ the trialitarian embedding associated to the twisted composition $(L,M,q,\cdot^{*2})$. The base change $\iota_L$ is the usual trialitarian embedding of $\on{Spin}(M,q)$ in a product of three orthogonal groups.  In particular, $\iota$ is injective.

Let $T$ be a maximal $k$-torus of $H$, and let $T_1=\iota(T)$. The center of the centralizer of $T_1$ in $R_{L/k}(\on{O}(M,q))$ is a $k$-torus $T_2$ containing $T_1$.

\begin{prop}\label{uniquetorus}
The torus $T_2$ is the unique maximal torus of $R_{L/k}(\on{O}(M,q))$ containing $T_1$.
\end{prop}

\begin{proof}
This can be checked over an algebraic closure of $k$, where $R_{L/k}(\on{O}(M,q))$ becomes a product of three orthogonal groups, $T_2$ becomes a product of three tori, and $\iota$ is the usual trialitarian embedding. Since $T$ projects isomorphically under each of the three projections, every factor of $T_2$ has rank at least $\on{rank}T=4$. We conclude that $T_2$ is a maximal torus in $R_{L/k}(\on{O}(M,q))$. Every maximal torus containing $T_1$ must contain the center of the centralizer of $T_1$ in $R_{L/k}(\on{O}(M,q))$, which by definition is $T_2$. This completes the proof.
\end{proof} 

\begin{defin}
Let $T$ be a maximal $k$-torus of $H$. We let $T'$ be the unique $L$-torus $U$ of $\on{O}(M,q)$ such that $T_2=R_{L/k}(U)$. 
\end{defin}

For the uniqueness statement implicit in the definition, see \cite[Lemma 4.2]{fiori2019rational} when $k$ is infinite and \cite[Remark 4.3]{fiori2019rational} when $k$ is finite.\footnote{For the proof of the main results of this paper, only the case when $k$ is infinite is relevant.} When $k$ is infinite, the torus $T'$ is the center of the centralizer of $T_2(k)\c \on{O}(M,q)(L)$.

The function $T\mapsto T'$ from the set of maximal $k$-tori in $H$ to the set of maximal $L$-tori in $O(M,q)$ is injective, but it is not necessarily surjective.

\section{Maximal tori and \'etale algebras with involution}\label{maxtori-etalesub}

Let $(E,\sigma)$ be an \'etale algebra with involution over $k$ such that $\on{rank}_kE=2\on{rank}_kE^{\sigma}=2n$, and set
\[\Phi:=\set{\phi\c \on{Hom}_k(E,\cl{k}): \phi\cap \phi\sigma=\emptyset, \phi\cup \phi\sigma= \on{Hom}_k(E,\cl{k})}.\] The action of $\on{Gal}(k)$ on $\on{Hom}(E,\cl{k})$ induces an action on $\Phi$. We let $E^{\Phi}$ be the \'etale algebra corresponding to $\Phi$. The action of $\sigma$ on $E$ induces an involution on $\on{Hom}_k(E,\cl{k})$ which commutes with the $\on{Gal}(k)$-action, which in turn induces an involution on $\Phi$. We denote by the same letter $\sigma$ the corresponding involution on $E^{\Phi}$, so that $(E^{\Phi},\sigma)$ is an \'etale algebra with involution. It is not difficult to show that $\on{rank}_kE^{\Phi}=2\on{rank}_k(E^{\Phi})^{\sigma}=2^n$. 

Let $L$ be a cubic \'etale $k$-algebra, and let $(E,\sigma)$ be an \'etale algebra with involution over $L$, such that $\on{rank}_LE=2\on{rank}_LE^{\sigma}=8$. By \cite[Corollary 18.28]{knus1998involutions} we have an isomorphism $L\otimes_kL\cong L\times (L\otimes_k\Delta)$, where $\Delta=\Delta(L)$ is the discriminant algebra of $L$; see \cite[p. 291]{knus1998involutions}. It follows that we have a $k$-isomorphism
\[(E,\sigma)\otimes_kL\cong (E,\sigma)\times (E_2,\sigma_2)\] where $(E_2,\sigma_2)$ is an \'etale algebra over $L\otimes_k \Delta$.

\begin{defin}\label{datum}
A \emph{datum} over $k$ is a quadruple $(L,E,\sigma,\beta)$, where $L$ is a cubic \'etale $k$-algebra, $(E,\sigma)$ is an \'etale $L$-algebra with involution such that $\on{rank}_LE=2\on{rank}_LE^{\sigma}=8$ and $\beta$ is an isomorphism \begin{equation}\label{beta-eq}\beta:(E,\sigma)\otimes_kL\xrightarrow{\sim} (E\times E^{\Phi},\sigma\times \sigma).\end{equation}
\end{defin}
There is an obvious notion of $k$-isomorphism of data over $k$, which we do not spell out explicitly.

Let $\alpha$ be a datum. We define the torus $T(\alpha)$ by setting
\begin{equation}\label{intrinsicF4}T(\alpha)(R):=\set{x\in (E\otimes_kR)^{\times}: x\sigma(x)=1,\Psi(x_2)=x_1\otimes 1} \end{equation} for every $k$-algebra $R$. Here we denote $\beta(x\otimes 1)=(x_1,x_2)$ for every $x\in (E\otimes_kR)^{\times}$. Furthermore, the homomorphism of $k$-algebras $\Psi:E^{\Phi}\to E\otimes_kE^{\Phi}$ is defined by considering the $\on{Gal}(k)$-equivariant homomorphism
\[\Psi(\sum_{\phi\in \Phi} a_\phi e_\phi):= \sum_{\phi_1\cap \phi_2=\set{\rho}} a_{\phi_1}a_{\phi_2}e_{\rho}\otimes( e_{\phi_1}+e_{\phi_2}+e_{\phi_1\sigma}+e_{\phi_2\sigma})\] over a separable closure of $k$, and then taking $\on{Gal}(k)$-invariants; see \cite[Definition 4.25]{fiori2019rational}. Using $\beta$, we may equivalently describe the $R$-points of $T(\alpha)$ as
\[T(\alpha)(R)=\set{(x_1,x_2)\in ((E\times E^{\Phi})\otimes_kR)^{\times}: x_i\sigma(x_i)=1, \Psi(x_2)=x_1}.\]

\begin{lemma}\label{relationD}
\begin{enumerate}[label=(\roman*)]
    \item\label{relationD1} Let $\Gamma=(L,M,q,\cdot^{*2})$ be a twisted composition over $k$, let $\iota:\on{Aut}(\Gamma)^{\circ}\hookrightarrow R_{L/k}(\on{O}(M,q))$ be the trialitarian embedding, and let $T$ be a maximal torus in $\on{Aut}(\Gamma)^{\circ}$. There exists a unique $\tau_q$-stable \'etale $L$-subalgebra $E$ of $\on{End}_L(M)$ such that $\on{rank}_LE=2\on{rank}_LE^{\tau_q}=8$ and such that $T=\iota^{-1}(R_{L/k}(U(E,\tau_q))$. Moreover, there is a canonical isomorphism \[\beta_{T}:(E,\sigma)\otimes_kL\xrightarrow{\sim}(E\times E^{\Phi},\sigma\times \sigma)\] which makes $\alpha=(L,E,\tau_q,\beta_{T})$ into a datum.
    \item\label{relationD2} Let $\alpha=(L,E,\sigma,\beta)$ be a datum. Then there exists a twisted composition $\Gamma=(L,M,q,\cdot^{*2})$ and an \'etale $L$-subalgebra $E'\c \on{End}_L(M)$ such that $(E,\sigma)\cong (E',\tau_q)$, $T:=\iota^{-1}(U(E',\tau_q|_{E'}))$ is a maximal torus in $\on{Aut}(\Gamma)^{\circ}$, and $\beta_{T}=\beta$.
\end{enumerate}
\end{lemma}

We will prove the lemma by invoking results of \cite[\S 4]{fiori2019rational}. Analogous properties are established there for trialitarian algebras. We first need to address the compatibility of our constructions with those of \cite{fiori2019rational}.

\begin{rmk}\label{compatible}
Let $\Gamma=(L,M,q,\cdot^{*2})$ be a twisted composition, and let $\Lambda=\on{End}(\Gamma)=(\on{End}_L(M),L,\tau_q,\eta)$ be the associated trialitarian algebra; see \cite[Example 43.7]{knus1998involutions}. Then $\on{Aut}(\Gamma)^{\circ}$ is canonically isomorphic to $\on{Spin}(\Lambda)$, where the latter group is the simply connected group associated to $\on{Aut}(\Gamma)$ (which is a twisted form of $\on{PGO}_8^+$). We also have a canonical $L$-isomorphism $\on{O}(M,q)\cong \on{O}(\on{End}_L(M),\tau_q)$. We obtain a commutative diagram
\[
\begin{tikzcd}
\on{Aut}(\Gamma)^{\circ} \arrow[r,hook,"\iota"] \arrow[d,equal] & R_{L/k}(\on{O}(M,q)) \arrow[d,equal] \\
\on{Spin}(\Lambda) \arrow[r,hook,"\iota'"] &  R_{L/k}(\on{O}(\on{End}_L(M),\tau_q)), 
\end{tikzcd}
\]
where $\iota$ and $\iota'$ are the trialitarian embeddings. Let $T$ be a maximal torus of $\on{Aut}(\Gamma)^{\circ}$. By \Cref{uniquetorus}, $\iota(T)$ is contained in a unique maximal $L$-torus of $R_{L/k}(\on{O}(M,q))$; let $(E,\sigma)$ be its type. If we view $T$ as a torus of $\on{Spin}(\Lambda)$, by \cite[Lemma 4.31]{fiori2019rational} there exists a unique maximal $L$-torus $U$ of $R_{L/k}(\on{O}(\on{End}_L(M),\tau_q))$ containing $\iota'(T)$. By transport of structure along the isomorphism $\Lambda\cong\on{End}(\Gamma)$, it is immediate to see that $U$ also has type $(E,\sigma)$.
\end{rmk}

\begin{proof}[Proof of \Cref{relationD}]

\ref{relationD1} Apply \cite[Lemma 4.31]{fiori2019rational} to $\on{End}(\Gamma)$ to construct $E$, then apply \cite[Lemma 4.32]{fiori2019rational} to construct $\beta_T$. 

\ref{relationD2} By \cite[Lemma 4.35]{fiori2019rational} applied to $(\on{End}_L(M),L,\tau_q)$, there exist a trialitarian algebra $\Lambda'=(\on{End}_L(M),L,\tau_q,\eta')$ and a maximal torus $T$ of $\on{Spin}(\Lambda')$ such that the procedure of \cite[Lemma 4.32, Lemma 4.35]{fiori2019rational} associates to $T$ the quadruple $(L,E,\sigma,\beta)$. By \cite[Proposition 44.16(1)]{knus1998involutions}, $\Lambda'\cong \on{End}(\Gamma')$ for some twisted composition $\Gamma'$. By \Cref{compatible}, there exists a maximal torus of $\on{Aut}(\Gamma')$ with datum $\alpha$.
\end{proof}

The following proposition gives the precise relation between maximal tori in groups of type $F_4$ and data $(L,E,\sigma,\beta)$. In particular, it shows that such data allow a uniform description of all maximal tori in groups of type $F_4$.

\begin{prop}\label{relation}
\begin{enumerate}[label=(\roman*)]
    \item\label{relation1} Let $J$ be an Albert $k$-algebra, let $G:=\on{Aut}(J)$, and let $T$ be a maximal torus of $G$. Let $\Gamma=(L,M,q,\cdot^{*2})$ be the twisted composition associated to $T$ in \Cref{constr}, and let $\iota:\on{Aut}(\Gamma)^{\circ}\hookrightarrow R_{L/k}(\on{O}(M,q))$ be the trialitarian embedding. There exists a unique $\tau_q$-stable \'etale $L$-subalgebra $E$ of $\on{End}_L(M)$ such that $\on{rank}_LE=2\on{rank}_LE^{\tau_q}=8$ and such that $T=\iota^{-1}(R_{L/k}(U(E,\tau_q))$. Moreover, there is a canonical isomorphism \[\beta_{T}:(E,\sigma)\otimes_kL\xrightarrow{\sim}(E\times E^{\Phi},\sigma\times \sigma)\] which makes $\alpha=(L,E,\tau_q,\beta_{T})$ into a datum.
    \item\label{relation2} Let $\alpha=(L,E,\sigma,\beta)$ be a datum. There exist a group $G$ of type $F_4$ and a maximal torus $T$ of $G$ such that the construction of \ref{relation1} applied to $T$ and $G$ gives a datum isomorphic to $\alpha$.
\end{enumerate}
\end{prop}
\begin{proof}
Combine \Cref{constr} with \Cref{relationD}.
\end{proof}

To complete the parallel with \cite{bayer2014embeddings}, we define the type of a maximal torus of a group of type $F_4$, and a notion of realizability.

\begin{defin}\label{type}
Let $J$ be an Albert algebra over $k$, let $G:=\on{Aut}(J)$, let $T$ be a maximal $k$-torus of $G$, and let $\alpha=(L,E,\sigma,\beta)$ be a datum. We say that $T$ is of \emph{type} $\alpha$ if the datum associated to $T$ in \Cref{relation}\ref{relation1} is isomorphic to $\alpha$. We say that $(G,\alpha)$ is \emph{realizable} if there exists a torus of type $\alpha$.
\end{defin}

\begin{rmk}\label{obvious}
Let $G$ be a group of type $F_4$, let $T$ be a maximal torus of $G$ of type $\alpha=(L,E,\sigma,\beta)$, and let $\Gamma=(L,M,q,\cdot^{*2})$ be the twisted composition associated to $T$ by \Cref{constr}. Then, by definition of type, $(E,q)$ is realizable in the sense of orthogonal groups.
\end{rmk}

\begin{lemma}\label{disc}
Let $\alpha=(L,E,\sigma,\beta)$ be a datum. Then:
\begin{enumerate}[label=(\roman*)]
    \item\label{disc1} $\on{disc}_L(E)=\on{disc}_k(L)$ in $L^{\times}/L^{\times 2}$;
    \item\label{disc2} The Clifford $L$-algebra $C(M,q)$ is split, that is, $(M,q)$ has trivial Clifford invariant.
\end{enumerate}
\end{lemma}
Here, we are identifying $\on{disc}_k(L)\in k^{\times}/k^{\times 2}$ with its image under the natural map $k^{\times}/k^{\times 2}\to L^{\times}/L^{\times 2}$. This map is injective; see \cite[Proposition 18.34]{knus1998involutions}.
\begin{proof}
 By \Cref{relation}\ref{relation2}, there exist a $k$-group $G$ of type $F_4$ and a maximal $k$-torus $T$ of $G$ of type $\alpha$. Let $\Gamma=(L,M,q,\cdot^{*2})$ be the twisted composition associated to $T$ by \Cref{constr}.
 
 \ref{disc1} By \Cref{obvious} $(E,q)$ is realizable. By \cite[Corollary 3.19]{fiori2019rational}, $\on{disc}(q)=\on{disc}_{k}(L)\in L^{\times}/L^{\times 2}$. By \cite[Lemma 1.3.2]{bayer2014embeddings} we have $\on{disc}(q)=\on{disc}_L(E)\in L^{\times}/L^{\times 2}$, hence $\on{disc}_L(E)=\on{disc}_k(L)\in L^{\times}/L^{\times 2}$.

\ref{disc2} If $L=k\times F$, the claim follows from \cite[Proposition 44.13]{knus1998involutions}. If $L$ is a field, the map $\on{Res}:\on{Br}_2(k)\to \on{Br}_2(L)$ given by tensoring by $L$ is injective, because the composition \[\on{Br}_2(k)\xrightarrow{\on{Res}} \on{Br}_2(L)\xrightarrow{\on{Cores}} \on{Br}_2(k)\] is given by $[A]\mapsto [A]^3=[A]$. Here $\on{Cores}$ is the corestriction map. Since $L\otimes_kL\cong L\times (L\otimes_k\Delta)$, the conclusion follows from the case when $L$ is not a field.
\end{proof}

One could describe $E^{\Phi}$ more explicitly. We refrain from doing so here, and instead refer the interested reader to \cite[\S 4.5.1]{fiori2019rational}.

\section{Local-global principle for tori in orthogonal groups}
For the proof of \Cref{mainthm}, we will need to make use of the main results in \cite{bayer2014embeddings}. In this section, we give an account of what we need.

The following observation will be useful in the sequel.
\begin{lemma}\label{switch}
Let $(E,\sigma)$ be an \'etale $k$-algebra with involution of rank $2n$. Assume that $E=R\times R$ for some \'etale algebra $R$, and that $\sigma$ acts by switching the two factors. Then $(E,q)$ is realizable if and only if $q\cong h_{2n}$ is hyperbolic.
\end{lemma}

\begin{proof}
One shows that the quadratic form $q_{\alpha}:x\mapsto \on{Tr}_{R/k}(\alpha x^2)$ is hyperbolic for every $\alpha\in E^{\sigma}=R$, and then applies \cite[Proposition 1.3.1]{bayer2014embeddings}. See the proof of \cite[Lemma 2.1.1]{bayer2014embeddings}.
\end{proof}

The next result characterizes maximal tori of an orthogonal group of type $D_n$ over a local field in terms of the cohomological invariants of $(E,\sigma)$. Recall that every such torus is of the form $U(E,\sigma)$ for some \'etale algebra with involution $(E,\sigma)$ such that $\on{rank}_kE=2\on{rank}_kE^{\sigma}=2n$. Following \cite[\S 2]{bayer2014embeddings}, we write $\Sigma_k$ for the set of places of $v$, and we let $\Sigma_k^{\on{split}}(E)$ be the set of $v\in \Sigma_k$ such that all places of $E^{\sigma}$ above $v$ split in $E$. We denote by $\rho_v$ the number of places of $E^{\sigma}$ above $v$ that ramify in $E$.

\begin{prop}\label{toriorthogonal}
Let $k$ be a global field, let $v$ be a place of $k$, let $(V,q)$ be a quadratic space of dimension $2n$ over $k$, and let $(E,\sigma)$ be an \'etale $k$-algebra with involution such that $\on{rank}_kE=2\on{rank}_kE^{\sigma}=2n$. 
\begin{enumerate}[label=(\alph*)]
    \item\label{toriorthogonal1} If $v\in \Sigma_k^{\on{split}}(E)$, then $(E_v,q_v)$ is realizable if and only if $q_v$ is hyperbolic.
    \item\label{toriorthogonal2} If $v\notin \Sigma_k^{\on{split}}(E)$ is a finite place, then $(E_v,q_v)$ is realizable if and only if $\on{disc}(q_v)=\on{disc}(E_v)$.
    \item\label{toriorthogonal3} If $v$ is a real place, then $(E_v,q_v)$ is realizable if and only if the signature of $q_v$ is of the form $(2r'+\rho_v,2s'+\rho_v)$ for some $r',s'\geq 0$.
\end{enumerate}
\end{prop}

\begin{proof}
\ref{toriorthogonal1} This is \cite[Lemma 2.1.1]{bayer2014embeddings}.

\ref{toriorthogonal2} If $(E_v,q_v)$ is realizable, then $\on{disc}(q_v)=\on{disc}(E_v)$ by \cite[Lemma 1.3.2]{bayer2014embeddings}. Assume that $\on{disc}(q_v)=\on{disc}(E_v)$. By \cite[Proposition 2.2.1]{bayer2014embeddings} there exists $\alpha\in (E^{\sigma})^{\times}$ such that $w(q_{\alpha})=w(q_v)$. By \cite[Lemma 1.3.2]{bayer2014embeddings}, we have $\on{disc}(q_{\alpha})=\on{disc}(E_v)$, hence $\on{disc}(q_{\alpha})=\on{disc}(q_v)$. It follows that $q_v$ and $q_{\alpha}$ have the same rank, discriminant and Hasse invariant, hence they are isometric. By \cite[Proposition 1.3.1]{bayer2014embeddings} $(E_v,q_{\alpha})$ is realizable, hence so is $(E_v,q_v)$.

\ref{toriorthogonal3} This is \cite[Proposition 2.3.2]{bayer2014embeddings}.
\end{proof}

Let $(E,\sigma)$ be an \'etale algebra of degree $2n$ with involution over $k$, and write $E=K_1\times\dots\times K_r$, where $K_i/k$ is a field extension, and set $I:=\set{1,\dots,r}$. We start by assuming that $\sigma(K_i)=K_i$ for every $i\in I$. For every $i\in I$, we let $F_i$ be the fixed field of $\sigma$ on $K_i$, so that $E^{\sigma}=F_1\times\dots\times F_r$ and $[K_i:F_i]=2$.\footnote{This is not the most general case, as it implies that $E$ has no factors $F\times F$, where $\sigma$ acts by switching the two factors. However, thanks to \Cref{noswitch}, considering this case will be enough for our purposes.} As in \cite[\S 1.3]{bayer2014embeddings}, we say that $(E,q)$ is realizable if $\on{O}(V,q)$ contains a maximal torus of the form $U(E,\sigma)$. 

Following \cite[\S 3.2]{bayer2014embeddings}, we let $\mc{C}(E,q)$ be the set of collections $(q_i^v)$ of quadratic spaces over $k_v$ satisfying conditions (i) - (iii) of \cite[Proposition 3.1.3]{bayer2014embeddings}:
\begin{enumerate}[label=(\roman*)]
    \item\label{bayer1} for all $i\in I$ and $v\in \Sigma_k$, $(K_v,q^v_i)$ is realizable;
    \item\label{bayer2} for all $v\in \Sigma_k$, $q_v\cong q^v_1\oplus\dots \oplus q^v_r$;
    \item\label{bayer3} for all $i\in I$, $w(q_i^v)=0$ for all but finitely many $v\in \Sigma_k$.
\end{enumerate}
For $C=(q_i^v)\in \mc{C}(E,q)$ and $i\in I$, set \[S_i(C):=\set{v\in \Sigma_k: w(q_i^v)=1}.\footnote{In \cite[\S 3.2]{bayer2014embeddings}, the symbol $\Sigma_k'$ is used instead of $\Sigma_k$. This is just a typographical error.}\] We say that $(q_i^v)\in \mc{C}(E,q)$ is connected if for all $i\in I$ such that $|S_i(C)|$ is odd, there exist $j\in I$ such that $j\neq i$ and $|S_j(C)|$ is odd, a sequence $i=i_1,\dots, i_m=j$ in $I$ such that for every $t=1,\dots,m-1$ there exists $v\in\Sigma_k$ that does not split in $K_{i_t}$ and in $K_{i_{t+1}}$. We say that $\mc{C}(E,q)$ is connected if it contains a connected element.

The following is \cite[Theorem 3.2.1]{bayer2014embeddings}. It plays an important role in the proof of \Cref{mainthm}, and in particular in the proof of \Cref{trivialclifford}.

\begin{thm}[Bayer-Fluckiger]
    Let $k$ be a global field of characteristic different from $2$, let $(V,q)$ be a $2n$-dimensional quadratic space over $k$, and let $(E,\sigma)$ be an \'etale algebra with involution such that $\on{rank}_kE=2\on{rank}_kE^{\sigma}=2n$. Then:
    \begin{enumerate}[label=(\alph*)]
        \item $(E,q)$ is realizable everywhere locally if and only if $\mc{C}(E,q)$ is non-empty;
        \item $(E,q)$ is realizable if and only if $\mc{C}(E,q)$ is connected.
    \end{enumerate}
\end{thm}

\section{Classification over local fields}
Let $k$ be a local field (of characteristic $\neq 2,3$), and let $G$ be a semisimple $k$-group of type $F_4$. We have $G=\on{Aut}(J)$ for some Albert $k$-algebra $J$. In this section we classify the maximal $k$-tori in $G$.

\subsection{Non-archimedean fields}

Assume first that $k$ is a local non-archimedean field. By \cite[\S 5.8 (iv)]{springer2000octonions}, there is a unique Albert algebra over $k$, that is, $J=J_s$ and $G=G_s$. In particular, $J=J(L,M)$ for every twisted composition $(L,M,q,\cdot^{*2})$ over $k$. The following proposition is now an immediate consequence of \Cref{relation}, but we record it for future reference.

\begin{prop}\label{nonarchimedean}
Let $k$ be a local non-archimedean field, let $G$ be the unique $k$-group of type $F_4$, and let $\alpha$ be a datum. Then $(G,\alpha)$ is realizable.
\end{prop}

\subsection{The field of real numbers}
Assume now that $k=\mathbb R$. We start by recalling the classification of twisted compositions over $\R$. The group $H^3(\R,\Z/3\Z)$ is $3$-torsion because so is $\Z/3\Z$, and is $2$-torsion because so is $\on{Gal}(\R)\cong \Z/2\Z$, hence $H^3(\R,\Z/3\Z)=0$. By \cite[Proposition 40.16]{knus1998involutions}, every twisted composition arises, up to similitude, from a symmetric composition of dimension $8$. By \cite[\S 1.10 (i)]{springer2000octonions}, the only two symmetric compositions of dimension $8$ are the split octonions $\O$ and the non-split octonions (Cayley numbers) $C$.

Twisted compositions over $\R^3$ and $\R\times\C$ are easily classified using \cite[\S 36.29]{knus1998involutions}. The classification has been written down explicitly in \cite[Example 3.23]{fiori2019rational}. If $L=\R^3$, then $M=(\O,\O,\O)$ or $(C,C,C)$. The quadratic form on $M$ is defined as the usual quadratic form on every factor $\O$ or $C$.
If $L=\R\times \C$, then $M=V\times (V\otimes\C)$, where $V$ is a quadratic space of signature $(1,7)$ or $(5,3)$.\footnote{In \cite{fiori2019rational}, the signatures are written with the negative part first, and the positive part last.}

There are three isomorphism classes of Albert $\R$-algebras: the split algebra $J_s$ defined in \Cref{constructionH}, 
\[J_1:=\set{\begin{pmatrix} \xi_1 & c_3 & \cl{c_2} \\ \cl{c_3} & \xi_2 & c_1 \\ c_2 & \cl{c_1} & \xi_3 \end{pmatrix}: \xi_i\in \R, c_i\in C},\]
and
\[J_c:=\set{\begin{pmatrix} \xi_1 & c_3 & -\cl{c_2} \\ \cl{c_3} & \xi_2 & c_1 \\ c_2 & -\cl{c_1} & \xi_3 \end{pmatrix}: \xi_i\in \R, c_i\in C};\] see \cite[\S 5.8 (ii)]{springer2000octonions}. 
If $J=J(L,M)$  the quadratic form $Q$ on $J$ is related to the quadratic form $q:M\to L$ by the formula
\begin{equation}\label{traceform}Q(l,m)=\on{Tr}_{L/k}(l^2)+\on{Tr}_{L/k}(q(m));\end{equation} see \cite[Theorem 38.6]{knus1998involutions}.
When $L=\R^3$, the previous formula reduces to \cite[(5.3)]{springer2000octonions}. In particular, using the above presentations of $J_s,J_1$ and $J_c$, we see that the signatures of the quadratic forms on $J_s,J_1$ and $J_c$ are $(15,12), (10,17)$ and $(27,0)$, respectively. 

Let $G_s:=\on{Aut}(J_s)$, $G_1:=\on{Aut}(J_1)$ and $G_c:=\on{Aut}(J_c)$. The groups $G_s,G_1$ and $G_c$ are representatives of the three isomorphism classes of $\R$-forms of type $F_4$. The group $G_s$ is split, $G_1$ has real rank $1$ (i.e. a maximal $\R$-split subtorus of $G_1$ has rank $1$), and $G_c$ is anisotropic (i.e. it does not contain any nontrivial $\R$-split subtorus).

\begin{lemma}\label{realalbert}
Let $\Gamma=(L,M,q,\cdot^{*2})$ be a twisted composition over $\R$, and let $J:=J(L,M)$. Assume first that $L=\R\times\R\times\R$.
\begin{enumerate}[label=(\alph*)]
    \item\label{realalbert1} If $M=(\O,\O,\O)$, then $J=J_s$.
    \item\label{realalbert2} If $M=(C,C,C)$ and $q$ is positive definite, then $J=J_c$.
    \item\label{realalbert3} If $M=(C,C,C)$ and $q$ is positive definite on one factor and negative definite in the other two, then $J=J_1$.
\end{enumerate}
Assume now that $L=\R\times \C$, and write $M=V\times W$, where $V$ is over $\R$ and $W$ is over $\C$.
\begin{enumerate}[label=(\alph*')]
    \item\label{realalbert4} If $(V,q|_V)$ has signature $(5,3)$, then $J=J_s$.
    \item\label{realalbert5} If $(V,q|_V)$ has signature $(7,1)$, then $J=J_1$.
\end{enumerate}
\end{lemma}

\begin{proof}
The proof proceeds by comparing the signatures of $(M,q)$ and $(J,Q)$ using (\ref{traceform}).

If $L=\R^3$, the signature of $(L,Q|_L)$ is $(3,0)$.

\ref{realalbert1} If $\Gamma$ is split, so is $J(L,M)$.

\ref{realalbert2} The signature of $(M,q)$ is $(24,0)$, hence the signature of $(J,Q)$ is $(27,0)$, and so $J=J_c$.

\ref{realalbert3} The signature of $(M,q)$ is $(8,16)$, hence the signature of $(J,Q)$ is $(10,17)$ and $J=J_1$.

If $(U,q')$ is a complex quadratic space of complex dimension $d$, then $(U,\on{Tr}_{\C/\R}(q'))$ is a real quadratic space of signature $(d,d)$. It follows that when $L=\R\times\C$, the signature of $(L,Q|_L)$ is $(2,1)$, and the signature of $(W,Q|_W)$ is $(8,8)$. 

\ref{realalbert4} If the signature of $(V,q|_V)$ is $(5,3)$, the signature of $(J,Q)$ is $(2,1)+(5,3)+(8,8)=(15,12)$, so $J=J_s$.


\ref{realalbert5} If the signature of $(V,q|_V)$ is $(1,7)$, the signature of $(J,Q)$ is $(10,17)$, so $J=J_1$.
\end{proof}

Recall that we denote by $\rho_v$ the number of field factors of $E^{\sigma}$ above $v$ that ramify in $E$. If $(E,\sigma)$ is an \'etale $k$-algebra with involution such that $\on{rank}_kE=\on{rank}_kE^{\sigma}=2n$, $(V,q)$ is a quadratic space of dimension $2n$ and $v$ real place of $k$, by \Cref{toriorthogonal}\ref{toriorthogonal3} $(E_v,q_v)$ is realizable if and only if the signature of $q_v$ is of the form $(\rho_v+2s,\rho_v+2r)$ for some $r,s\geq 0$. We want to apply this result to $k=L$. Note that $L_v$ is never a field, but we may apply the result on each factor of $L_v$ separately.

\begin{prop}\label{realcond}
Let $k$ be a global field, let $v$ be a real place of $k$, let $G$ be a $k$-group of type $F_4$, and let $\alpha=(L,E,\sigma,\beta)$ be a datum. Then $(G_v,\alpha_v)$ is realizable over $k_v$ if and only if one of the following is true:
\begin{enumerate}[label=(\roman*)]
    \item\label{realcond1} $G_v$ is split and either $L_v=\R^3$ and $\rho_v$ is even, or $L_v=\R\times\C$ and $\rho_v$ is odd;
    \item\label{realcond2} $G_v$ is anisotropic, $L_v=\R^3$ and $\rho_v=0$;
    \item\label{realcond3} $G_v$ has real rank $1$ and either $L_v=\R^3$ and $\rho_v$ is even, or $L_v=\R\times \C$ and $\rho_v=1$.
\end{enumerate}
\end{prop}

\begin{proof}
\ref{realcond1} When $G_v$ is split, by \Cref{realalbert}\ref{realalbert1}, $\on{SO}_8^3$ is a subgroup of $G_v$. Since $0\leq \rho_v\leq 4$ and the quadratic form defining $\on{SO}_8$ has signature $(4,4)$, it follows from \Cref{toriorthogonal}\ref{toriorthogonal3} that $(G_v,\alpha_v)$ is realizable if and only if $\rho_v$ is even.

\ref{realcond2} When $G_v$ is anisotropic, by \Cref{realalbert} $L_v=\R^3$. Furthermore, the signature of $(M_v,q_v)$ is $(24,0)$, hence by \Cref{toriorthogonal}\ref{toriorthogonal3} $(G_v,\alpha_v)$ is realizable if and only if $\rho_v=0$.

\ref{realcond3} When $G_v$ has real rank $1$, by \Cref{realalbert} either $L_v=\R^3$ and $(M,q)$ has signature $(8,16)$, or $L_v=\R\times \C$ and the real factor of $(M_v,q_v)$ has signature $(7,1)$. By \Cref{toriorthogonal}\ref{toriorthogonal3}, in the first case $(G_v,\alpha_v)$ is realizable if and only if $\rho_v$ is even, and in the second $(G_v,\alpha_v)$ is realizable if and only if $\rho_v\leq 1$.

By \Cref{realalbert}, we have $L_v=\R^3$ when $H_v=\on{Spin}_8$ or $\on{Spin}(8)$, and we have $L_v=\R\times \C$ when $H_v=\on{Spin}(7,1)$ or $\on{Spin}(3,5)$. 
\end{proof}

We conclude the section by describing the maximal tori in real forms of type $F_4$ explicitly, that is, without reference to a datum $\alpha$. We will not use this result in the sequel.

\begin{prop}\label{realtori}
Let $k=\R$, let $G$ be a simple $\R$-group of type $F_4$, and let $T$ be an $\R$-torus of rank $4$. Then $T$ embeds in $G$ if and only if $\on{rank}_{\R}T\leq \on{rank}_{\R}G$, that is, if and only if one of the following is true:
\begin{enumerate}[label=(\alph*)]
    \item\label{realtori1} $G=G_s$;
    \item\label{realtori2} $G=G_c$ and $T=R^{(1)}_{\C/\R}(\G_m)^4$;
    \item\label{realtori3} $G=G_1$ and $T$ is isomorphic to one of \[\G_m\times R^{(1)}_{\C/\R}(\G_m)^3,\quad R_{\C/\R}(\G_m)\times R^{(1)}_{\C/\R}(\G_m)^2,\quad R^{(1)}_{\C/\R}(\G_m)^4.\]
\end{enumerate}
\end{prop}

\begin{proof}
    \ref{realtori1} We know that $G_s$ contains a subgroup isomorphic to $\on{Spin}_8$. We have an embedding $\on{Spin}_4\times \on{Spin}_4\hookrightarrow \on{Spin}_8$. Recall the exceptional isomorphism $\on{Spin}_4\cong \on{SL}_2\times \on{SL}_2$. It follows that every torus of the form $\G_m^a\times R^{(1)}_{\C/\R}(\G_m)^{4-a}$ arises as a maximal torus of $\on{Spin}_8$. 
    
    Consider the embedding $\on{Spin}(3,1)\times \on{Spin}(1,3)\hookrightarrow \on{Spin}_8$. We have the isomorphisms $\on{Spin}(3,1)\cong \on{Spin}(1,3)\cong R_{\C/\R}(\on{SL}_2)$, hence $\on{Spin}(3,1)$ and $\on{Spin}(1,3)$ contain subtori isomorphic to $R_{\C/\R}(\G_m)$. It follows that $R_{\C/\R}(\G_m)^2$ is a subtorus of $\on{Spin}_8$.
    
    Consider the embedding $\on{Spin}_2\times \on{Spin}_6\hookrightarrow \on{Spin}_8$. We have $\on{Spin}_2\cong \G_m$ and $\on{Spin}_6\cong \on{SL}_4$. The maximal tori in $\on{SL}_4$ have the form $R^{(1)}_{E/\R}(\G_m)$, where $E$ is an \'etale algebra of degree $4$ over $\R$. The subtori of $\on{SL}_4$ corresponding to $E=\C\times \C$ and $E=\C\times \R^2$ are $R_{\C/\R}(\G_m)\times R^{(1)}_{\C/\R}(\G_m)$ and $R_{\C/\R}(\G_m)\times \G_m$, respectively. It follows that $\on{Spin}_8$ contains subtori isomorphic to $\G_m\times R_{\C/\R}(\G_m)\times R^{(1)}_{\C/\R}(\G_m)$ and $R_{\C/\R}(\G_m)\times \G_m^2$.
    
    By \Cref{realalbert}\ref{realalbert4}, $G_s$ contains a subgroup isomorphic to $\on{Spin}(3,5)$. We have an embedding $\on{Spin}(3,1)\times \on{Spin}(0,4)$. We have seen that $\on{Spin}(3,1)$ contains a subtorus isomorphic to $R_{\C/\R}(\G_m)$. Since $\on{Spin}(0,4)$ is anisotropic, all its maximal subtori are isomorphic to $R^{(1)}_{\C/\R}(\G_m)^2$. It follows that $R_{\C/\R}(\G_m)\times R^{(1)}_{\C/\R}(\G_m)^2$ is a subgroup of $\on{Spin}(3,5)$.

    \ref{realtori2} Since $G_c$ is anisotropic, $T$ is also anisotropic, hence $T\cong R^{(1)}_{\C/\R}(\G_m)^4$.
    
    \ref{realtori3} Since $G_1$ has real rank $1$, only the listed possibilities can occur. Let $H$ be the connected component of the automorphism group of the twisted composition of \Cref{realalbert}\ref{realalbert3}. Then $H$ embeds in the anisotropic group $R_{\R^3/\R}(\on{SO}(8))=\on{SO}(8)^3$, hence it is anisotropic. It follows that $R^{(1)}_{\C/\R}(\G_m)^4$ embeds in $G_1$.
    
    By \Cref{realalbert}\ref{realalbert5}, $G_1$ contains a subgroup isomorphic to $\on{Spin}(7,1)$. We have an embedding $\on{Spin}(3,1)\times\on{Spin}(4,0)\hookrightarrow \on{Spin}(7,1)$. We have seen that  $\on{Spin}(3,1)$ contains a subtorus isomorphic to $R_{\C/\R}(\G_m)$, and that all maximal tori of $\on{Spin}(4,0)$ are isomorphic to $R^{(1)}_{\C/\R}(\G_m)^2$. It follows that $R_{\C/\R}(\G_m)\times R^{(1)}_{\C/\R}(\G_m)^2$ is a subgroup of $\on{Spin}(7,1)$, hence of $G_1$. 
    
    We have an embedding $\on{Spin}(1,1)\times \on{Spin}(6,0)\hookrightarrow \on{Spin}(7,1)$. Since $\on{Spin}(1,1)=\on{Spin}_2\cong \G_m$ and $\on{Spin}(6,0)$ is anisotropic, it follows that $\on{Spin}(7,1)$ contains a subtorus isomorphic to $\G_m\times R^{(1)}_{\C/\R}(\G_m)^3$.
\end{proof}

\section{Local-global properties of twisted compositions}
The purpose of this section is to prove two auxiliary results on twisted compositions over global fields.

\begin{prop}\label{localglobalHG}
Let $k$ be a global field, let $J$ be an Albert algebra over $k$, let $(L,M,q,\cdot^{*2})$ be a twisted composition over $k$. If $J(L_v,M_v)\cong J_v$ for every $v\in\Sigma_k$, then $J(L,M)\cong J$.
\end{prop}

\begin{proof}
For every field extension $K/k$, isomorphism classes of twisted compositions over $K$ correspond to elements of $H^1(K,\on{Spin}_8\rtimes S_3)$; see \cite[Proposition 36.7]{knus1998involutions}. Isomorphism classes of Jordan algebras over $K$ correspond to elements of $H^1(K,G_s)$, where $G_s$ is the split group of type $F_4$. Under these identifications, the map $H^1(K,\on{Spin}_8\rtimes S_3)\to H^1(K,G_s)$ induced by the inclusion $\on{Spin}_8\rtimes S_3\hookrightarrow G_s$ sends a twisted composition over $K$ to its Springer construction; see \cite[Corollary 38.7]{knus1998involutions}.

Applying the Galois cohomology functor to the inclusion $\on{Spin}_8\rtimes S_3\hookrightarrow G_s$ for $K=k$ and $K=k_v$ for every $v\in \Sigma_k$ yields a commutative diagram
\begin{equation}\label{twisteddiag}
\begin{tikzcd}
H^1(k,\on{Spin}_8\rtimes S_3) \arrow[r] \arrow[d] & H^1(k,G_s) \arrow[d,hook,"j"] \\
\prod_{v\in\Sigma_k} H^1(k_v,\on{Spin}_8\rtimes S_3) \arrow[r] & \prod_{v\in\Sigma_k}H^1(k_v,G_s). 
\end{tikzcd}
\end{equation}
Since $G_s$ is simply connected, by the Hasse principle the map $j$ is injective. The conclusion follows from the commutativity of the diagram.
\end{proof}

\begin{prop}\label{twistedcomp-locglob}
Let $k$ be a global field, let $L$ be a cubic \'etale algebra over $k$, and let $(M,q)$ be a quadratic space of dimension $8$ over $L$. Assume that for every place $v$ of $k$, $(L_v,M_v,q_v)$ admits the structure of a twisted composition $(L_v,M_v,q_v,\cdot^{*2}_v)$ over $k_v$. Then $(L,M,q)$ admits a unique structure of twisted composition $(L_v,M_v,q_v,\cdot^{*2})$ over $k$ whose base change to $k_v$ is isomorphic $(L_v,M_v,q_v,\cdot^{*2}_v)$. 
\end{prop}

\begin{proof}
We start by showing uniqueness. For every field extension $K/k$, twisted compositions over $K$ are classified by $H^1(K,\on{Spin}_8\rtimes S_3)$, hence we need to show that the natural map \[H^1(k,\on{Spin}_8\rtimes S_3)\xrightarrow{\partial_v} \prod_v H^1(k_v,\on{Spin}_8\rtimes S_3)\] is injective. Let $[\Gamma],[\Gamma']\in H^1(k,\on{Spin}_8\rtimes S_3)$ be such that $\partial_v([\Gamma])=\partial_v([\Gamma'])$ for every $v\in\Sigma_k$. We have a commutative diagram
\begin{equation}\label{twistedcomp1}
\begin{tikzcd}
H^1(k,\on{Spin}_8\rtimes S_3) \arrow[r,"\pi"] \arrow[d,"\partial_v"] & H^1(k,S_3) \arrow[d]\\
\prod_v H^1(k_v,\on{Spin}_8\rtimes S_3) \arrow[r,"\pi_v"] & \prod_v H^1(k_v,S_3).
\end{tikzcd}
\end{equation}
The vertical map on the right is injective. It follows that there exists $[L]\in H^1(k,S_3)$ such that $\pi([\Gamma])=\pi([\Gamma'])=[L]$. Then $[\Gamma],[\Gamma']\in \pi^{-1}([L])=H^1(k, \prescript{L}{}{\on{Spin}_8})$ and $[\Gamma_v],[\Gamma'_v]\in \pi^{-1}([L_v])=H^1(k_v,\prescript{L_v}{}{\on{Spin}_8})$ for every $v\in\Sigma_k$. Since $\prescript{L}{}{\on{Spin}_8}$ is simply connected, by the Hasse principle for torsors the map \[H^1(k,\prescript{L}{}{\on{Spin}_8})\to \prod_vH^1(k_v,\prescript{L_v}{}{\on{Spin}_8})\] is injective. This map is just the restriction to $H^1(k,\prescript{L}{}{\on{Spin}_8})$ of $\partial_v$, hence it sends $[\Gamma]$ and $[\Gamma']$ to the same element. We conclude that $[\Gamma]=[\Gamma']$, that is, $\Gamma\cong \Gamma'$.

We now turn to existence. By \cite[(ii) p. 108]{springer2000octonions}, every twisted composition over a global field is reduced (the result is stated only for number fields, but the proof works equally well in the function field case). This means that, up to similitude, every twisted composition over a global field arises from a symmetric composition. Let $r$ be the number of real places of $k$. By \cite[(v) p. 22]{springer2000octonions}, every $r$-tuple of symmetric compositions at the real places of $k$ comes from a symmetric composition over $k$. This gives a twisted composition $\Gamma$ over $k$ such that $\Gamma_v$ is similar to $(L_v,M_v,q_v,\cdot^{*2})$ for every real place $v$. This means that $\Gamma_v\cong (L_v,M_v,\epsilon_vq_v,\delta_v\cdot_v^{*2})$, where $\epsilon_v\in L_v^{\times}/L_v^{\times 2}$ and $\delta_v\in L_v^{\times}/L_v^{\times 2}$. Since $v$ is real, we have $k_v\cong \R$ and either $L_v=\R^3$ or $L_v=\R\times \C$. By the Weak Approximation theorem, it is not difficult to find $\epsilon\in L$ and $\delta\in k$ such that the $\epsilon_v$ and $\delta_v$ come from $\epsilon$ and $\delta$. If we set $\Gamma':=(L,M,\epsilon q,\delta\cdot^{*2})$, we have $\Gamma_v\cong \Gamma'_v$ for every real place $v$. On the other hand, if $v$ is a finite place, then $H^1(k_v,\prescript{L_v}{}{\on{Spin}_8})$ is trivial because $\prescript{L}{}{\on{Spin}_8}$ is simply connected. This means that the map $H^1(k_v,\on{Spin}_8\rtimes S_3)\to H^1(k_v,S_3)$ is a bijection, that is, a twisted composition over $k_v$ is uniquely determined by the associated cubic \'etale $k_v$-algebra. We conclude that  $\Gamma_v\cong \Gamma'_v$ for every $v\in\Sigma_k$, hence $\Gamma\cong \Gamma'$, by the uniqueness part.
\end{proof}

\section{Local-global principle for maximal tori in orthogonal groups with trivial Clifford invariant}
The purpose of this section is to prove the following proposition, which will be used in the proof  of \Cref{mainthm} and is also of independent interest.

\begin{prop}\label{trivialclifford}
Let $k$ be a global field, let $(V,q)$ be a quadratic space of dimension $2n$, let $(E,\sigma)$ be an \'etale $k$-algebra with involution such that $\on{rank}_kE=2\on{rank}_kE^{\sigma}=2n$. Assume that $q$ has trivial Clifford invariant. If $(E_v,q_v)$ is realizable for every $v\in\Sigma_k$, then $(E,q)$ is realizable.
\end{prop}

Let $(V,q)$ be a quadratic space of dimension $2n$. The condition that $q$ has trivial Clifford invariant (i.e. the Clifford algebra $C(V,q)$ is split), is equivalent to the following condition on the Hasse invariant of $q$:
  \begin{equation}\label{hasse-clifford}
    w(q)=
    \begin{cases}
     (-1,\on{disc}(q)), & \text{if}\ n\equiv 0\pmod{4} \\
     0, & \text{if}\ n\equiv 1\pmod{4} \\
     (-1,-\on{disc}(q)), & \text{if}\ n\equiv 2\pmod{4} \\
     (-1,-1), & \text{if}\ n\equiv 3\pmod{4}.
    \end{cases}
  \end{equation}
In particular, we see that if $q$ and $q'$ are $2n$-dimensional quadratic forms with trivial Clifford invariant and equal discriminant, then $w(q)=w(q')$. These formulas are proved in \cite[Chapter V, 3.20]{lam2005introduction}, where they are expressed in terms of $\on{det}(q)$ instead of $\on{disc}(q)$. However, $\on{det}(q)$ appears only when $n$ is even (i.e. $\on{dim}(q)$ is divisible by $4$), and in that case $\on{disc}(q)=\on{det}(q)$.

\begin{lemma}\label{noswitch}
In the course of proving \Cref{trivialclifford}, we may assume that $E=K_1\times\dots\times K_r$, where every $K_i$ is a field and is $\sigma$-stable.
\end{lemma}

\begin{proof}
Let $(E,\sigma)$ be an \'etale algebra with involution over $k$, and write \[E=E_1\times E_2,\qquad E_1:=K_1\times\dots\times K_s,\qquad E_2:=K_{s+1}^2\times\dots\times K_r^2,\] where every $K_i$ is a field, $K_i$ is $\sigma$-invariant when $1\leq i\leq s$ and every $\sigma$ acts by switching the two factors $K_i$ when $s+1\leq i\leq r$. We let $\on{rank}_kE_1=2n_1$ and $\on{rank}_kE_2=2n_2$.

Let $(V,q)$ be a quadratic space of dimension $2n$, and assume that $(E,q)$ is realizable everywhere locally. By \cite[Theorem 3.2.1(a)]{bayer2014embeddings}, the set $\mc{C}(E,q)$ is not empty, and by \cite[Theorem 3.2.1(b)]{bayer2014embeddings} and the assumptions $\mc{C}(E_1,q')$ is connected for every quadratic space $(V',q')$ of dimension $2n_1$ and trivial Clifford invariant.  

We have $\Sigma_k^{\on{split}}(E_2)=\Sigma_k$. By \Cref{toriorthogonal}\ref{toriorthogonal1}, for every collection $(q_i^v)\in \mc{C}(E,q)$, we have that $q_i^v$ is hyperbolic for every $v\in \Sigma_k$ and every $s+1\leq i\leq r$. It follows that for every place $v$ of $k$, we have a decomposition $q_v=q'_v\perp h_{2n_2}$ for some $q'_v$ of dimension $2n_1$. By Witt's Decomposition Theorem \cite[Chapter I, 4.1]{lam2005introduction}, we may write $q=q_a\perp h_{2m}$, where $q_a$ is anisotropic and $h_{2m}$ is hyperbolic of dimension $2m$; the quadratic forms $q_a$ and $h$ are uniquely determined up to isometry. By the Hasse-Minkowski Principle \cite[Chapter VI, 3.1]{lam2005introduction} there exists a place $v$ such that $(q_a)_v$ is anisotropic. We have $(q_a)_v\perp h_{2m}\cong q_v\cong q'_v\perp h_{2n_2}$, hence by Witt's Cancellation Theorem \cite[Chapter I, 4.2]{lam2005introduction} we must have $m\geq n_2$ (and $q'_v\cong (q_a)_v\perp h_{2(m-n_2)}$).

Since $m\geq n_2$, there exists a quadratic form $p$ such that $q\cong p\perp h_{2n_2}$. We have $p_v\perp h_{2n_2}\cong q\cong q^1_v\perp\dots\perp q^s_v\perp h_{2n_2}$, hence by Witt's Cancellation Theorem $p_v\cong q^1_v\perp\dots\perp q^s_v$. This shows that $\mc{C}(E_1,p)$ is non-empty, i.e., that $(E_1,p)$ is realizable everywhere locally. Since $\on{dim}(p)$ and $\on{dim}(h_{2n_2})$ are even and $q$ has trivial Clifford invariant, by \cite[Chapter V, (3.13)]{lam2005introduction}, the Clifford invariant of $p$ equals $(\on{disc}(p),1)$, and so is also trivial. By assumption, this means that $(E_1,p)$ is realizable. By \Cref{switch}, $(E_2,h_{2n_2})$ is realizable. Since $(E,\sigma)=(E_1,\sigma|_{E_1})\times (E_2,\sigma|_{E_2})$, and $q\cong p\perp h_{2n_2}$, we conclude that $(E,q)$ is realizable, as desired.
\end{proof}

\begin{lemma}\label{square}
Let $k$ be an arbitrary field, let $q_1$ and $q_2$ be quadratic forms of dimension $2n_1$ and $2n_2$. Assume that $\on{disc}(q_1)$ is a square and that $q_1$ and $q_2$ have trivial Clifford invariant. Then $q_1\perp q_2$ has trivial Clifford invariant.
\end{lemma}

\begin{proof}
Since $\dim(q_1)$ and $\on{dim}(q_2)$ are even, by \cite[Chapter V, (3.13)]{lam2005introduction} the Clifford invariant of $q_1\perp q_2$ equals $(\on{disc}(q_1),\on{disc}(q_2))$. Since $\on{disc}(q_1)$ is a square, $(\on{disc}(q_1),\on{disc}(q_2))=0$, and the conclusion follows.
\end{proof}

\begin{lemma}\label{mod4}
Let $q_1$ and $q_2$ be quadratic forms over $\R$, and let $(r_1,s_1)$ and $(r_2,s_2)$ be their signatures. Assume that $s_1\equiv s_2\pmod{2}$ and that $w(q_1)=w(q_2)$. Then $s_1\equiv s_2\pmod{4}$.
\end{lemma}

\begin{proof}
We have $w(q_i)=\frac{1}{2}{s_i(s_i-1)}(-1,-1)$, hence $w(q_1)=w(q_2)$ is equivalent to $s_1^2-s_1\equiv s_2^2-s_2\pmod{4}$, i.e. $(s_1-s_2)(s_1+s_2-1)\equiv 0\pmod{4}$. By assumption $s_1\equiv s_2\pmod{2}$, hence $s_1+s_2-1$ is odd. We conclude that $4$ divides $s_1-s_2$.
\end{proof}

\begin{lemma}\label{explicittrivialclifford}
Let $(E,\sigma)$ be an \'etale $k$-algebra with involution such that $\on{rank}_kE=2\on{rank}_kE^{\sigma}=2n$. There exists a quadratic form $q$ of dimension $2n$ and trivial Clifford invariant such that $(E,q)$ is realizable.
\end{lemma}

\begin{proof}
We may write $E^{\sigma}=k[t]/(f(t))$, for some separable polynomial $f(t)\in k[t]$, in such a way that $E=E^{\sigma}(\sqrt{t})$ and $\sigma(\sqrt{t})=-\sqrt{t}$. Let $q:E\to k$ be the quadratic form defined by $q(x):=\on{Tr}_{E/K}(\alpha x\sigma(x))$, where $\alpha=\frac{1}{2f'(t)}$. By \cite[Proposition 1.3.1]{bayer2014embeddings}, $(E,q)$ is realizable. By \cite[Lemma 3.6]{fiori2012characterization}, $q$ has trivial Clifford invariant.\footnote{In \cite{fiori2012characterization}, the Clifford invariant of $q$ is called the Witt invariant, and is denoted by $W(q)$.}
\end{proof}

\begin{proof}[Proof of \Cref{trivialclifford}]
By \Cref{noswitch}, we may assume that $E=K_1\times\dots\times K_r$, where every $K_i$ is a field and is $\sigma$-stable, and $\on{rank}_kE=2\on{rank}_kE^{\sigma}=2n$. Let $I:=\set{1,\dots,r}$. If $I$ is connected, then the claim follows from \cite[Theorem 3.2.1]{bayer2014embeddings}, so assume that $I$ is not connected. Rearranging the $K_i$ if necessary, we may then find some $1<s<r$ such that if $i\leq s$ and $j\geq s+1$, then $i$ and $j$ are not connected. Write \[E=E_1\times E_2,\qquad E_1:=K_1\times\dots\times K_s,\qquad E_2:=K_{s+1}\times\dots\times K_r.\]
We have $\on{rank}_kE_i=2\on{rank}_kE_i^{\sigma}=2n_i$, for some $n_i\geq 1$. For every real place $v$ and $i=1,2$, we denote by $\rho^i_v$ the number of real places of $E_i^{\sigma}$ over $v$ which do not ramify in $E_i$. We have $\rho_v=\rho^1_v+\rho^2_v$ for every real place $v$.

By \Cref{explicittrivialclifford}, there exist quadratic forms $q_i$ over $k$ of dimension $2n_i$ and with trivial Clifford invariant such that $(E_i,q_i)$ are realizable, for $i=1,2$. By \cite[Proposition 2.4.1]{bayer2014embeddings} we have $\on{disc}(q_i)=\on{disc}(E_i)\in k^{\times}/k^{\times 2}$, and for every real place $v$ the signature of $q_i^v$ is of the form $(2r^i_v+\rho^i_v,2s^i_v+\rho^i_v)$, where $r^i_v,s^i_v\geq 0$. 

By \cite[Proposition 2.4.1]{bayer2014embeddings} we also have $\on{disc}(q)=\on{disc}(E)$, $w(q_v)=w(h_{2n})$ when $v\in \Sigma_k^{\on{split}}$, and for every real place $v$ the signature of $q_v$ is $(2r_v+\rho_v,2s_v+\rho_v)$, for some $r_v,s_v\geq 0$. 

Let $v$ be a place of $k$. Assume that neither $\on{disc}(q_1)$ nor $\on{disc}(q_2)$ are squares in $k_v^{\times}$. We have \[\on{disc}(q_1)=\prod_{i=1}^s\on{disc}(K_i),\qquad \on{disc}(q_2)=\prod_{j=s+1}^r\on{disc}(K_j).\]
It follows that there exist $1\leq i\leq s$ and $s+1\leq j\leq r$ such that $\on{disc}(K_i)$ and $\on{disc}(K_j)$ are not squares in $k_v^{\times}$. This means that $i$ and $j$ are connected, a contradiction. 

We have just shown that for every place $v$ of $k$, either $\on{disc}(q_1)$ or $\on{disc}(q_2)$ is a square in $k_v^{\times}$. By \Cref{square}, this implies that $q_1\perp q_2$ has trivial Clifford invariant. Since $q_1$ and $q_2$ are even-dimensional, we have $\on{disc}(q)=\on{disc}(q_1)\on{disc}(q_2)=\on{disc}(q_1\perp q_2)$. It follows from (\ref{hasse-clifford}) that $w(q_v)=w((q_1\perp q_2)_v)$ for every place (finite or archimedean) $v$.

What we have shown so far is that $q$ and $q_1\perp q_2$ have the same rank, discriminant and Hasse invariant. It is not necessarily the case that the signatures of $q$ and $q_1\perp q_2$ coincide. Our strategy is to modify $q_i$ so that $(E_i,q_i)$ remains realizable for $i=1,2$, discriminant and Hasse invariant remain the same, but the signatures at real places also agree.

Let $v$ be a real place. We have $\rho_v=\rho_v^1+\rho_v^2$, hence the signature of $q_1\perp q_2$ is $(2r^1_v+2r^2_v+\rho_v,2s^1_v+2s^2_v+\rho_v)$. In particular, the signatures of $q$ and $q_1\perp q_2$ are congruent modulo $2$. By \Cref{square}, the negative parts of the signatures of $q$ and $q_1\perp q_2$ are congruent modulo $4$, and it clearly follows that the same is true for the positive parts. In other words, for every real place $v$ we may write
\[\on{sign}(q_v)=2(r_v,s_v)+(\rho_v,\rho_v), \qquad \on{sign}((q_1\perp q_2)_v)=2(r'_v,s'_v)+(\rho_v,\rho_v),\] where $r_v\equiv r'_v\pmod 2$ and $s_v\equiv s'_v\pmod 2$. Note that whenever two real quadratic forms have signatures that differ by multiples of $4$, their discriminants are equal. Since the two signatures of $q_v$ and $(q_1\perp q_2)_v$ differ by multiples of $4$, we can modify the signature of $q_1$ and $q_2$ by multiples of $4$ to obtain the signature of $q$, leaving the discriminant untouched. Here are the details. 

For every real place $v$, we construct two pairs $(a^1_v,b^1_v)$ and $(a^2_v,b^2_v)$ of non-negative integers such that $a^i_v+b^i_v=n_i-\rho^i_v$ and such that $a^1_v+a^2_v=2r_v-\rho_v$, $b^1_v+b^2_v=2s_v-\rho_v$, $a^1_v+a^2_v\equiv r_v\pmod 2$ and $b^1_v+b^2_v\equiv s_v\pmod 2$ . We start by setting $a^i_v:=r^i_v$ and $b^i_v:=s^i_v$, and then we modify them according to the following procedure. If $a^1_v+a^2_v=r_v$ and $b^1_v+b^2_v=s_v$, then $\on{sign}(q_v)=\on{sign}((q_1\perp q_2)_v)$ and we stop. Suppose that $a^1_v+a^2_v>r_v$ (if $a^1_v+a^2_v>r_v$, then $b^1_v+b^2_v>s_v$, and the procedure is entirely analogous). Since $a^1_v+a^2_v\equiv r_v\pmod 2$, we have $a^1_v+a^2_v-2\geq r_v$. In particular, either (i) $a^1_v\geq 2$, (ii) $a^1_v\geq 2$ or (iii) $a^1_v=a^2_v=1$. In case (i) we modify $(a^1_v,b^1_v)$ into $(a^1_v-2,b^1_v+2)$, in case (ii) we modify $(a^2_v,b^2_v)$ into $(a^2_v-2,b^2_v+2)$ and in case (iii) we modify $(a^1_v,b^1_v)$ into $(a^1_v-1,b^1_v+1)$ and $(a^2_v,b^2_v)$ into $(a^2_v-1,b^2_v+1)$, respectively. This process eventually stops, because the difference $a^1_v+a^2_v-r_v$ is non-negative, even, and decreases by $2$ at each step. At each step of the process the sums $a^i+b^i_v$ remain invariant, hence $a^i_v+b^i_v=r^i_v+s^i_v=n_i-\rho^i_v$. At the end of the process we have $a^1_v+a^2_v=r_v$, hence $b^1_v+b^2_v=s_v$ also holds.

 For every place $v$ and $i=1,2$, consider the $2n_i$-dimensional quadratic form $p_i^v$ defined as follows: if $v$ is finite then $p_i^v:=(q_i)_v$, and if $v$ is real then $p_i^v$ has signature $2(a_v^i,b^v_i)+(\rho_v^i,\rho_v^i)$. By \cite[Theorem 6.10]{scharlau1985quadratic}, there exist quadratic forms $p_1$ and $p_2$ such that $(p_i)_v=p_i^v$ for every place $v$. The torus $(E_i,p_i)$ is realizable everywhere locally: if $v$ is finite then this is because $((E_i)_v,(q_i)_v)$ is realizable, and if $v$ is real it follows from \cite[Proposition 2.3.1]{bayer2014embeddings}. By induction, we deduce that $(E_i,p_i)$ is realizable over $k$. If $v$ is a finite place, all invariants of $p_i$ and $q_i$ agree, hence $q_v\cong (q_1\perp q_2)_v\cong (p_1\perp p_2)_v$. If $v$ is a real place, then $q_v\cong(p_1\perp p_2)_v$ because they have the same signature. By the weak Hasse-Minkowski principle \cite[Chapter VI, 3.3]{lam2005introduction}, it follows that $q\cong p_1\perp p_2$. Since $(E_i, p_i)$ are realizable, $(E,\sigma)=(E_1,\sigma)\times (E_2,\sigma)$, and $q\cong p_1\perp p_2$, we conclude that $(E,q)$ is realizable. 
\end{proof}

\section{Proofs of \Cref{mainthm} and \Cref{classification}}

\begin{proof}[Proof of \Cref{mainthm}]
Let $k$ be the a global field, let $J$ be an Albert algebra over $k$, let $G:=\on{Aut}(J)$ and let $\alpha=(L,E,\sigma,\beta)$ be a datum. It is clear that if $(G,\alpha)$ is realizable over $k$, then $(G_v,\alpha_v)$ is realizable over $k_v$ for every place $v$ of $k$.

Assume that $(G_v,\alpha_v)$ is realizable for every $v\in\Sigma_k$. By definition, for every $v\in\Sigma_k$ there exists a maximal torus $T_v$ of $G_v$ of type $\alpha_v$. Let $\Gamma_v:=(L_v,M_v,q_v,\cdot_v^{*2})$ be the twisted composition associated to $T_v$ by \Cref{constr}. Note that $L_v=L\otimes_kk_v$ for every $v\in\Sigma_k$, but $M_v,q_v,\cdot^{*2}_v$ are not assumed to be defined over $k$.
 
By \Cref{obvious},  $(E_v,q_v)$ is realizable for every $v\in\Sigma_k$. By \Cref{toriorthogonal} and \Cref{disc}\ref{disc1}, we have $\on{disc}(q_v)=\on{disc}_{k_v}(L_v)$ in $L_v^{\times}/L_v^{\times 2}$. Since $L$ is defined over $k$, we deduce that the collection $(\on{disc}(q_v))$ comes from a global discriminant $\delta=\on{disc}_k(L)$. By \Cref{disc}\ref{disc2}, the Clifford invariant of $\on{O}(M_v,q_v)$ is trivial for every $v$. By (\ref{hasse-clifford}) we see that $w(q_v)=(-1,\on{disc}(q_v))=(-1,\delta)$. In particular, the number of places $v$ such that $w(q_v)\neq 0$ is finite and of even cardinality, by the Hilbert reciprocity law. By \cite[Theorem 6.10]{scharlau1985quadratic}, there exists a quadratic form $q$ over $L$ such that $q_v$ is the completion of $q$ at $v$ for every place $v$. It follows that the triples $(L_v,M_v,q_v)$ all arise as base change of a single $(L,M,q)$. By \Cref{twistedcomp-locglob}, there exists a unique twisted composition $\Gamma=(L,M,q,\cdot^{*2})$ over $k$ such that $\Gamma\times_kk_v\cong \Gamma_v$ for every $v\in\Sigma_k$. 
 
 We have $J_v\cong J(L_v,M_v)$ for every $v\in\Sigma_k$, hence by \Cref{localglobalHG} we deduce that $J\cong J(L,M)$. It follows from \Cref{springerconstr} that $H:=\on{Aut}(\Gamma)^{\circ}$ embeds in $G$. 
 
 By construction, for every $v\in\Sigma_k$ the pair $(E_v,q_v)$ is realizable: there is a unique $\tau_q$-stable $E'_v\c \on{Aut}(L_v,M_v)$ such that $(E'_v,\tau_{q_v})\cong (E_v,\sigma_v)$ and $T_v=U(E'_v,\tau_{q_v})$. By \Cref{disc}\ref{disc2}, the $L$-space $(M,q)$ has trivial Clifford invariant. By \Cref{trivialclifford}, this implies that $(E,q)$ is realizable over $k$: there exists $E'\c \on{Aut}_L(M)$ that is $\tau_q$-stable and such that $(E',\tau_q)\cong (E,\sigma)$. Let $\Lambda=(\on{End}_L(M),L,\tau_q,\eta)$ be the trialitarian algebra given by applying \cite[Lemma 4.35]{fiori2019rational} to $D=\on{End}_L(M)$ and $\tau=\tau_q$. Note that by \cite[Proposition 44.16]{knus1998involutions}, $\Lambda$ arises from a twisted composition $\Gamma'=(L,M,q,\cdot'^{*2})$.
 
 By \cite[Lemma 4.35]{fiori2019rational}, there exists a maximal torus of $\on{Spin}(\Lambda)$ such that the unique maximal $L$-torus of $R_{L/k}(\on{O}(\on{End}_L(M),\tau_q))$ containing the image of $U$ under the the trialitarian embedding is of type $(E,\sigma)$. Let $\beta':(E,\sigma)\otimes L \to (E\times E^{\Phi},\sigma\times\sigma)$ be the induced isomorphism. 
 
 Since $k$ is a global field, by \cite[Corollary 3.30, Lemma 4.35]{fiori2019rational}, $\on{End}(\Gamma)\cong\on{End}(\Gamma')$, that is, $\Lambda$ is isomorphic to the trialitarian algebra $\on{End}(\Gamma)$ associated to $\Gamma$. By \cite[Lemma 4.32]{fiori2019rational}, $\beta'$ is determined by $\Gamma$ and $(E,\sigma)$, hence $\beta=\beta'$. By \Cref{compatible}, we conclude that there exists a maximal torus $U$ of $\on{Aut}(\Gamma)^{\circ}$ such that the unique maximal $L$-torus of $R_{L/k}(\on{O}(M,q))$ containing the image of $U$ under the the trialitarian embedding is of type $(E,\sigma)$, and such that the induced isomorphism $(E,\sigma)\otimes L \to (E\times E^{\Phi},\sigma\times\sigma)$ is $\beta$. By definition, this means that $G$ contains a maximal torus of type $\alpha$, as desired.
\end{proof}

\begin{proof}[Proof of \Cref{classification}]
If $(G,\alpha)$ is realizable, then $(G_v,\alpha_v)$ is realizable for every $v\in\Sigma_k$. The condition on real places is then satisfied by \Cref{realcond}. 

Assume that the conditions at real places are satisfied. By \Cref{mainthm}, it suffices to show that $(G_v,\alpha_v)$ is realizable for every $v\in \Sigma_k$. If $v$ is a finite place, $(G_v,\alpha_v)$ is realizable by \Cref{nonarchimedean}. If $v$ is a real place, $(G_v,\alpha_v)$ is realizable by \Cref{realcond}.
\end{proof}

\section{Conjugacy classes of maximal tori}\label{sec:conjclasses}

We conclude this work with a study of the conjugacy classes of maximal tori in groups of type $F_4$ over $k$. This is not used anywhere else in the paper.

\begin{prop}\label{conjugation}
Let $k$ be a field of characteristic different from $2$ and $3$, let $G$ be a $k$-group of type $F_4$, let $T_1$, $T_2$ be two maximal $k$-tori in $G$, and let $H_i$ be the unique subgroup of type $D_4$ containing $T_i$ given by \Cref{constr}, for $i=1,2$. Then $T_1$ and $T_2$ are conjugate in $G$ if and only if there exists an isomorphism $\phi:H_1\to H_2$ sending $T_1$ to $T_2$.
\end{prop}

\begin{proof}
If there is $g\in G(k)$ such that $gT_1g^{-1}=T_2$, then by \Cref{constr}\ref{constr3} we have $gH_1g^{-1}=H_2$, hence the automorphism of $G$ given by conjugation by $g$ sends $H_1$ to $H_2$ and $T_1$ to $T_2$.

Assume that there exists an isomorphism $\phi:H_1\to H_2$ sending $T_1$ to $T_2$. We claim that $\phi$ extends to an automorphism of $G$. By \Cref{constr}\ref{constr3}, we may write $H_i=\on{Aut}(\Gamma_i)^{\circ}$ for some twisted compositions $\Gamma_i=(L_i,M_i,q_i,\cdot^{*2}_i)$ and $i=1,2$. By \cite[Theorem 44.8]{knus1998involutions}, $\phi$ induces an isomorphism $f:\Gamma_1\xrightarrow{\sim} \Gamma_2$. If we let $J$ be the Albert algebra such that $G=\on{Aut}(J)$, we have $J(L_1,M_1)\cong J(L_2,M_2)\cong J$. By \cite[Proof of Corollary 38.7]{knus1998involutions}, $f$ extends to an automorphism of $J$, which by  \cite[Theorem 44.8]{knus1998involutions} induces an automorphism of $G$ extending $\phi$, as claimed. By \cite[Theorem 25.16]{knus1998involutions}, every automorphism of $F_4$ is inner, hence $T_1$ and $T_2$ are conjugate.
\end{proof}

Let $k$ be a global field of characteristic different from $2$ and $3$, let $J$ be an Albert algebra over $k$, let $G:=\on{Aut}(J)$ and let $T_1$ and $T_2$ be maximal tori of $G$. For $i=1,2$, let $\alpha_i=(L_i,E_i,\tau_{q_i},\beta_i)$ be the associated datum of $T_i$, and let $\Gamma_i=(L_i,M_i,q_i,\cdot_i^{*2})$ be the twisted composition associated to $T_i$ in \Cref{constr}, so that we have the Springer decompositions $J=L_i\perp M_i$. Combining \Cref{conjugation} with \cite[Corollary 4.63]{fiori2019rational}, we conclude that $T_1$ and $T_2$ are conjugate in $G$ if and only if:
\begin{enumerate}[label=(\alph*)]
    \item there exists an isomorphism $\phi:\Gamma_1\xrightarrow{\sim}\Gamma_2$;
    \item the $L$-torus $\phi_1(U(E_1,\tau_{q_1}))$ is conjugate to $U(E_2,\tau_{q_2})$ in $\on{SO}(M_2,q_2)$, where $\phi_1:\on{SO}(M_1,q_1)\to \on{SO}(M_2,q_2)$ is the isomorphism induced by $\phi$;
    \item the $L$-torus $\phi_2(U(E_1^{\Phi_1},\tau_{q_1}))$ is conjugate to $U(E_2^{\Phi_2},\tau_{q_2})$, where we denote by $\phi_2:\on{Spin}(M_1,q_1)\to \on{Spin}(M_2,q_2)$ the isomorphism induced by $\phi$.
\end{enumerate}
Condition (b) is characterized in terms of $(L_i,E_i,\tau_{q_i})$ in \cite[Theorem 4.14]{fiori2019rational}, and condition (c) is characterized in \cite[Theorem 4.29]{fiori2019rational}. Note that under the projections $\on{Spin}(M_i,q_i)\to \on{SO}(M_i,q_i)$, the $L_i$-tori $U(E_i^{\Phi_i},\tau_{q_i})$ map onto $U(E_i,\tau_{q_i})$; see \cite[Theorem 4.24]{fiori2019rational}.

\section{Acknowledgements}
The first author would like to thank Andrei Rapinchuk for the initial encouragement to pursue this problem years ago.

The second author thanks Eva Bayer-Fluckiger for useful correspondence, and Julia Gordon for helpful comments and encouragement.

\end{document}